\documentclass[10pt]{amsart}
\usepackage{tikz-cd}
\usepackage{amssymb}
\usepackage{mathrsfs}
\usepackage[shortlabels]{enumitem}
\usepackage[colorlinks,linkcolor=black,citecolor=black,urlcolor=black]{hyperref}

\numberwithin{equation}{section}
\newtheorem{thm}[subsection]{Theorem}

\newtheorem{lem}[subsection]{Lemma}
\newtheorem{prop}[subsection]{Proposition}
\theoremstyle{definition}
\newtheorem{df}[subsection]{Definition}
\newtheorem{rmk}[subsection]{Remark}
\newtheorem{exm}[subsection]{Example}
\newtheorem{const}[subsection]{Construction}

\newcommand{\bE}{\mathbf{E}}

\newcommand{\bK}{\mathbf{K}}

\newcommand{\A}{\mathbb{A}}

\newcommand{\F}{\mathbb{F}}
\newcommand{\G}{\mathbb{G}}

\newcommand{\N}{\mathbb{N}}
\renewcommand{\P}{\mathbb{P}}
\newcommand{\Q}{\mathbb{Q}}

\newcommand{\Z}{\mathbb{Z}}

\newcommand{\cA}{\mathcal{A}}

\newcommand{\cC}{\mathcal{C}}
\newcommand{\cD}{\mathcal{D}}
\newcommand{\cE}{\mathcal{E}}
\newcommand{\cF}{\mathcal{F}}
\newcommand{\cG}{\mathcal{G}}

\newcommand{\cI}{\mathcal{I}}

\newcommand{\cL}{\mathcal{L}}
\newcommand{\cM}{\mathcal{M}}

\newcommand{\cO}{\mathcal{O}}
\newcommand{\cP}{\mathcal{P}}

\newcommand{\rD}{\mathrm{D}}

\newcommand{\rH}{\mathrm{H}}

\newcommand{\rM}{\mathrm{M}}
\newcommand{\rN}{\mathrm{N}}

\newcommand{\rQ}{\mathrm{Q}}

\newcommand{\rU}{\mathrm{U}}

\newcommand{\sT}{\mathscr{T}}

\DeclareMathOperator{\Hom}{Hom}
\DeclareMathOperator{\Spec}{Spec}

\newcommand{\colim}{\mathop{\mathrm{colim}}}
\newcommand{\colimPrL}{\mathop{\mathrm{colim}^{\mathrm{Pr}^\mathrm{L}}}}
\newcommand{\limPrR}{\mathop{\mathrm{lim}^{\mathrm{Pr}^\mathrm{R}}}}
\newcommand{\id}{\mathrm{id}}
\newcommand{\ul}{\underline}
\newcommand{\ol}{\overline}

\newcommand{\pt}{\mathrm{pt}}
\newcommand{\lSm}{\mathrm{lSm}}
\newcommand{\eSm}{\mathrm{eSm}}
\newcommand{\lSch}{\mathrm{lSch}}
\newcommand{\Sm}{\mathrm{Sm}}

\newcommand{\sat}{\mathrm{sat}}
\newcommand{\gp}{\mathrm{gp}}
\newcommand{\Ayo}{\mathrm{Ayo}}
\newcommand{\red}{\mathrm{red}}
\newcommand{\Var}{\mathrm{Var}}

\newcommand{\Sh}{\mathrm{Sh}}
\newcommand{\Sp}{\mathrm{Sp}}
\newcommand{\Mod}{\mathrm{Mod}}

\newcommand{\Sch}{\mathrm{Sch}}
\newcommand{\ad}{\mathrm{ad}}
\newcommand{\Ex}{\mathrm{Ex}}
\newcommand{\PrL}{\mathrm{Pr}^\mathrm{L}}

\newcommand{\CAlg}{\mathrm{CAlg}}
\newcommand{\bDelta}{\mathbf{\Delta}}
\newcommand{\op}{\mathrm{op}}

\newcommand{\SH}{\mathrm{SH}}

\newcommand{\UT}{\mathrm{U}\mathscr{T}}
\newcommand{\QUT}{\mathrm{QU}\mathscr{T}}

\newcommand{\st}{\mathrm{st}}
\newcommand{\ex}{\mathrm{ex}}
\newcommand{\DA}{\mathrm{DA}}
\newcommand{\RigSH}{\mathrm{RigSH}}
\newcommand{\RigDA}{\mathrm{RigDA}}
\newcommand{\DM}{\mathrm{DM}}

\newcommand{\setale}{\mathrm{s\acute{e}t}}

\newcommand{\et}{\mathrm{\acute{e}t}}

\newcommand{\unit}{\mathbf{1}}

\newcommand{\KGL}{\mathbf{KGL}}

\begin{document}
\title{Log motivic nearby cycles}
\author{Doosung Park}
\address{Department of Mathematics and Informatics, University of Wuppertal, Germany}
\email{dpark@uni-wuppertal.de}
\date{\today}
\begin{abstract}
We define the log motivic nearby cycles functor. We show that this sends the motive of a proper smooth scheme over the fraction field of a DVR to the motive of the boundary of a log smooth model assuming absolute purity, which is unconditional in the equal characteristic case. In characteristic $0$, we show that the $\infty$-categories of motives over the standard log point and rigid analytic motives are equivalent, and we relate log motivic nearby cycles functor with Ayoub's motivic nearby cycles functor.
\end{abstract}
\maketitle

\section{Introduction}

For simplicity of exposition,
we restrict to the DVR case in the introduction even though we can formulate the results in a more general context.
Let $\cO_K$ be a DVR with fraction field $K$ and residue field $k$.
Let
\[
\Psi^{\Ayo}
\colon
\SH(K)
\to
\SH(k)
\]
denote Ayoub's nearby cycles functor in \cite[D\'efinition 3.5.6]{Ayo07}.
This is a motivic version of Grothendieck's nearby cycles functor in \cite[Expos\'e XII]{MR0354657},
which aims to understand the limit behavior at the ``boundary'' $\Spec(k)$ of $\Spec(K)$.

\

Even though $\Psi^{\Ayo}$ is defined without assuming that $\cO_K$ has equal characteristic $0$,
many theorems about $\Psi^{\Ayo}$ assume this, e.g.,
\cite[Th\'eor\`eme 3.3.46]{Ayo07}, \cite[Scholie 1.3.26]{MR3381140}, \cite{MR3131490}, and \cite{MR3742195}, or are proven for \'etale motives, e.g., \cite[Th\'eor\`eme 10.19]{AyoubEtale}.
A fundamental reason for this assumption is that
a log smooth model does not need to exist except in equal characteristic $0$ even assuming resolution of singularities,
see Proposition \ref{lognearby.8}, Remark \ref{lognearby.9}, and \cite[Proposition 3.3.40, Remarque 3.3.41]{Ayo07}.
When working with the \'etale topology,
one can instead use the de Jong alterations in the mixed characteristic case,
see the proof of \cite[Th\'eor\`eme 10.9]{AyoubEtale}.

\

The analogue of $\Psi^{\Ayo}$ for Grothendieck rings of varieties is 
the motivic reduction map
\[
\mathrm{MR}
\colon
\bK(\Var_K)
\to
\bK(\Var_k)
\]
due to Nicaise-Shinder \cite[Proposition 3.2.1]{zbMATH07088056},
whose construction was inspired by the motivic nearby fiber due to Denef-Loeser \cite{zbMATH01944722} and the motivic volume map due to Hrushovski-Kazhdan \cite{zbMATH05234057}.
Again, the equal characteristic $0$ assumption was used for the construction of $\mathrm{MR}$.

\

The purpose of this paper is to explore the motivic limit behavior using logarithmic geometry.
Let $\cO_K^\dagger$ be the log ring $\cO_K$ with the standard log structure,
and let $k^\dagger$ be the log ring $(k,\N\oplus k^*)$,
whose spectrum is the standard log point.
The \emph{log motivic nearby cycles functor} is
\[
\Psi^{\log}:=i^*j_*\colon \SH(K)\to \SH(k^\dagger),
\]
where $i\colon \Spec(k^\dagger)\to \Spec(\cO_K^\dagger)$ and $j\colon \Spec(K)\to \Spec(\cO_K^\dagger)$ are the obvious strict immersions.
We refer to \cite[3.1.5]{MR1622751} and \cite[(8.2.1)]{MR1922832} for the log nearby cycles functors in the non-motivic context.

\

For a proper smooth scheme $X$ over $K$,
a \emph{log smooth model of $X$} is a proper log smooth log scheme $Y$ over $\cO_K^\dagger$ such that $Y-\partial Y\simeq X$.
In this case,
$Y$ is vertical over $\cO_K^\dagger$.
We have noted that a log smooth model (with projective $X$) does not need to exist except in equal characteristic $0$.
However,
in the equal characteristic case,
the six-functor formalism of fs log schemes in \cite{logsix} allows us to compute $\Psi^{\log}(M_K(X))$ in terms of a log smooth model of $X$ whenever it exists.
The precise result is as follows.

\begin{thm}
[See Proposition \ref{lognearby.2}(3) and Theorem \ref{lognearby.3}]
\label{intro.1}
Let $X$ be a proper smooth scheme over $K$.
If $\cO_K$ has equal characteristic $p$ and $Y$ is a log smooth model of $X$,
then there is an isomorphism in $\SH(k^\dagger)$
\[
\Psi^{\log}(M_K(X))
\simeq
M_{k^\dagger}(Y\times_{\Spec(\cO_k^\dagger)}\Spec(k^\dagger)).
\]
In particular, $M_{K^\dagger}(Y\otimes_{\Spec(\cO_k^\dagger)}\Spec(k^\dagger))$ is independent of the choice of $Y$.
\end{thm}

In the mixed characteristic case,
we need to assume absolute purity in the sense of \cite[Definition 3.5]{regGysin} as follows since this is not automatically satisfied unlike the equal characteristic case.
We need absolute purity since this implies that $\Psi^{\log}$ sends the unit to the unit.
We refer to \cite[Construction 2.2.5]{logshriek} for the notation $\Mod_\bE$ below.

\begin{thm}
[See Proposition \ref{lognearby.2}(2) and Theorem \ref{lognearby.3}]
\label{intro.2}
Assume that $\Spec(\cO_K)$ is a scheme over a base scheme $B$.
Let $X$ be a proper smooth scheme over $K$,
and let $\bE$ be a commutative algebra object of $\SH(B)$ satisfying absolute purity.
If $Y$ is a log smooth model of $X$,
there is an isomorphism in $\Mod_{\bE}(k^\dagger)$
\[
\Psi^{\log}(M_K(X))
\simeq
M_{k^\dagger}(Y\times_{\Spec(\cO_k^\dagger)}\Spec(k^\dagger)).
\]
\end{thm}

By \cite[Theorem 3.6(2)]{MR4225026},
the rational motivic sphere spectrum $\unit_\Q\in \SH(B)$ satisfies absolute purity.
Hence Theorem \ref{intro.2} holds for $\DA(-,\Q)$ by \cite[5.3.35]{CD12}.
D\'eglise \cite[Conjecture A in Example 1.3.4]{MR3930052} conjectured that the motivic sphere spectrum $\unit\in \SH(B)$ satisfies absolute purity.
If this holds,
then Theorem \ref{intro.2} holds for $\SH$ too.

\

To extract the motivic information of $\Psi^{\log}(M_K(X))$ in terms of schemes,
it is desired to have a nice functor $\SH(k^\dagger)\to \SH(k)$.
There exists no morphism of fs log schemes $o\colon \Spec(k)\to \Spec(k^\dagger)$.
Nevertheless, we construct a functor $o^*$ as follows that behaves like a pullback functor.

\begin{thm}
[See Theorem \ref{nearby.4}]
\label{intro.3}
There exists a colimit preserving symmetric monoidal functor of symmetric monoidal $\infty$-categories
\[
o^*\colon \SH(k^\dagger)
\to
\SH(k)
\]
satisfying the following two properties:
\begin{enumerate}
\item[\textup{(1)}]
$o^*\pi^*\simeq \id$,
where $\pi\colon \Spec(k^\dagger)\to \Spec(k)$ is the projection.
\item[\textup{(2)}]
$o^*w_r^*\simeq o^*$ for every $r\in \N^+$,
where $w_r\colon \Spec(k^\dagger)\to \Spec(k^\dagger)$ is the morphism induced by the multiplication $r\colon \N\to \N$.
\end{enumerate}
\end{thm}

Now, the \emph{motivic nearby cycles functor} is defined to be the composite functor
\[
\Psi
:=
o^*\Psi^{\log}
\colon
\SH(K)
\to
\SH(k),
\]
whose source and target categories involve no log structures.
To show that this behaves well,
we provide the following result.
See the beginning of \S \ref{description} for the notation $(\ul{V},\ul{D})$ below.

\begin{thm}
[See Theorem \ref{description.3}]
\label{intro.6}
Let $V:=(\ul{V},\ul{D})\to \Spec(\cO_K^\dagger)$ be a proper vertical saturated log smooth morphism,
where $\ul{D}=\ul{D_1}+\cdots+\ul{D_n}$ is a strict normal crossing divisor such that each $\ul{D_i}$ is connected.
For every nonempty subset $I$ of $\{1,\ldots,n\}$,
we set
\[
\ul{D_I}
:=
\bigcap_{i\in I} \ul{D_i},
\text{ }
\ul{D_I^\circ}
:=
\ul{D_I}-\bigcup_{j\notin I} \ul{D_j},
\]
and let $F_I^\circ$ be the fiber of the induced morphism of normal bundles $\rN_{D_I^\circ}(\ul{V}-\bigcup_{j\notin I}\ul{D_j})\to \rN_S (\A_S^1)\simeq \A_S^1$ at $\{1\}$.
Then there is a natural isomorphism in $\sT(S)$
\[
\Psi (M_K(V-\partial V))
\simeq
\colim_{I\subset \{1,\ldots,n\},I\neq \emptyset}
M_S(F_I^\circ),
\]
where the colimit runs over the category of nonempty subsets of $\{1,\ldots,n\}$.
\end{thm}

Together with Remark \ref{description.5},
we can compute $[\Psi(M_K(X))]\in K_0(\SH_c(k))$ for a proper smooth scheme $X$ in terms of $X\times_{\Spec(\cO_K^\dagger)}\Spec(k^\dagger)$ if $\cO_K$ has equal characteristic and a log smooth model $Y$ of $X$ exists,
where $\SH_c(k)$ is the $\infty$-category of constructible motives,
and $K_0(-)$ is the Grothendieck group of a stable $\infty$-category.
We refer to \cite[Theorem 3.1]{MR3131490} and \cite[Theorem 8.6]{MR3742195} for a similar result in the equal characteristic $0$ case.

\

We have the following comparison result of the two motivic nearby cycles functors in the equal characteristic $0$ case with $K=k(\! (x) \! )$.

\begin{thm}
[See Theorem \ref{comp.1}]
\label{intro.5}
Let $k$ be a field of characteristic $0$.
If $K=k(\! (x) \! )$,
then there is a natural isomorphism
\[
\Psi
\simeq
\Psi^\Ayo.
\]
\end{thm}

Ayoub expected that for a field $k$ of characteristic $0$,
the categories of motives \cite{MR3381140} over $k^\dagger$ and rigid analytic motives over $k(\! (x) \! )$,
which have very different origins,
are equivalent,
which we prove as follows.
This indicates that cohomology theories of fs log schemes and rigid analytic varieties are closely related at least in characteristic $0$.

\begin{thm}
[See Theorem \ref{quasi-uni.5}]
\label{intro.4}
Let $k$ be a field of characteristic $0$.
Then there are equivalences of $\infty$-categories
\[
\SH(k^\dagger)
\simeq
\RigSH(
k(\! (x) \! )
),
\text{ }
\DA(k^\dagger,\Lambda)
\simeq
\RigDA(
k(\! (x) \! )
,\Lambda),
\]
where $\Lambda$ is a commutative ring.
\end{thm}

With the \'etale topology and $\Q$-coefficients, such a comparison might be achievable for the equal and mixed characteristic cases too, 
see Remark \ref{comp.2} for the known results in this direction.

\

In summary, $\Psi$ agrees with $\Psi^{\Ayo}$ in the above equal characteristic $0$ case,
and $\Psi$ has advantages in the equal positive and mixed characteristic cases due to Theorems \ref{intro.1}--\ref{intro.6}.

\subsection*{Notation and conventions}

Our standard reference for log geometry is Ogus's book \cite{Ogu}.
We assume that every fs log scheme in this paper has a Zariski log structure.
A finite dimensional noetherian fs log scheme is an fs log scheme whose underlying scheme is finite dimensional noetherian.
Let $\cP$ be one of the following properties of morphisms of schemes: separated, proper, and projective.
We say that a morphism of fs log schemes is $\cP$ if its underlying morphism of schemes is $\cP$.
We employ the following notation throughout this paper:

\begin{tabular}{l|l}
$B$ & a finite dimensional noetherian separated base scheme
\\
$\Sch$ & the category of finite dimensional noetherian separated schemes
\\
$\lSch$ &  the category of finite dimensional noetherian separated
\\
& fs log schemes
\\
$\lSm$ & the class of log smooth morphisms
\\
$\eSm$ & the class of exact log smooth morphisms
\\
$\Hom_{\cC}$ & the hom space in an $\infty$-category $\cC$
\\
$\hom_{\cC}$ & the hom spectra in a stable $\infty$-category $\cC$
\\
$\sT$ & a compactly generated log motivic $\infty$-category
\\
$\id \xrightarrow{\ad} f_*f^*$ & the unit of an adjunction $(f^*,f_*)$
\end{tabular}

\subsection*{Acknowledgements}

We thank Joseph Ayoub for his expectation on the equivalence between $\infty$-categories of motives over the standard log point and rigid analytic motives.
We also thank Federico Binda for many helpful comments.
This research was conducted in the framework of the DFG-funded research training group GRK 2240: \emph{Algebro-Geometric Methods in Algebra, Arithmetic and Topology}.

\section{Log motivic nearby cycles functors}

In this paper,
we fix a log motivic $\infty$-category $\sT$ in the sense of  \cite[Definition 2.1.1]{logshriek},
which is a dividing Nisnevich sheaf of symmetric monoidal presentable $\infty$-categories
\[
\sT\in \Sh_{dNis}(\lSch/B,\CAlg(\PrL))
\]
satisfying certain conditions.
It is called \emph{compactly generated} if $M_S(X)(d)[n]$ is compact in $\sT(S)$ for all $X\in \lSm/S$ and $d,n\in \Z$.
We also assume this for $\sT$.

A fundamental example of a compactly generated log motivic $\infty$-category is $\SH$ defined in \cite[Definition 2.5.5]{logA1},
see \cite[Theorem 2.1.2]{logshriek}.

Recall that $B$ is a finite dimensional noetherian base scheme.
If $\bE$ is a commutative algebra object of $\SH(B)$,
then $\Mod_\bE$ is a compactly generated log motivic $\infty$-category too,
see \cite[Theorem 2.2.6]{logshriek}.

We work with $\sT^\ex$ in \cite[Definition 2.3.1]{logshriek} for the six-functor formalism.
For $S\in \lSch/B$,
$\sT^\ex(S)$ is the full subcategory of $\sT(S)$ generated under colimits by $M_S(X)(d)[n]$ for $X\in \eSm/S$ and $d,n\in \Z$.
If $f\colon X\to S$ is a morphism in $\lSch/B$,
then we have the functor $f^*\colon \sT^\ex(S)\to \sT^\ex(X)$ and its right adjoint $f_*$.
Since we assume that $\sT$ is compactly generated,
\cite[Proposition 2.4.6]{logsix} implies that $f_*$ preserves colimits.
If $f\colon X\to S$ is in $\eSm$,
then $f^*$ admits a left adjoint $f_\sharp$.
We refer to \cite[Theorems 1.2.1, 1.3.1]{logsix} for a summary of the six-functor formalism for $\sT^\ex$.

Recall from \cite[Definition 2.5.1]{logshriek} that $\sT^\st(S)$ is the full subcategory of $\sT(S)$ generated under colimits by $M_S(X)(d)[n]$ for strict smooth $X$ over $S$ and $d,n\in \Z$.

If the rank of $\ol{\cM}_{S,s}^\gp$ is $\leq 1$ for every point $s\in S$,
then $\eSm/S=\lSm/S$ by \cite[Proposition I.4.2.1(4)]{Ogu},
so we have $\sT^\ex(S)=\sT(S)$.

\begin{exm}
\label{lognearby.5}
We have the following examples for $\bE\in \SH(B)$.
\begin{enumerate}
\item[(1)]
The motivic sphere spectrum $\unit\in \SH(B)$ yields $\SH\simeq \Mod_{\unit}$.
\item[(2)]
The classical Eilenberg-MacLane spectrum $\rH \Lambda\in \SH(B)$ for a commutative ring $\Lambda$ and $X\in \Sch/B$ yields
\[
\DA(X,\Lambda)\simeq \Mod_{\rH \Lambda}(X).
\]
The rational motivic sphere spectrum $\unit_\Q\in \SH(B)$ is isomorphic to $\rH \Q$,
so we have an equivalence
\[
\DA(X,\Q)\simeq \Mod_{\unit_\Q}(X)
\]
for $X\in \Sch/B$,
see \cite[5.3.35]{CD12}.
\item[(3)]
The motivic Eilenberg-MacLane spectrum $\rM \Lambda\in \SH(B)$ for a commutative ring $\Lambda$ yields
\[
\DM(X,\Lambda)
\simeq
\Mod_{\rM \Lambda}(X)
\]
by \cite[Theorem 3.1]{MR3404379} if $X$ is a regular scheme in $\Sch/k$,
$k$ is a field,
and the exponential characteristic of $k$ is invertible in a commutative ring $\Lambda$.
Furthermore,
by \cite[Theorem 16.2.22]{CD12},
we have
\[
\DA_{\et}(X,\Q)
\simeq
\Mod_{\rM \Q}(X).
\]
\end{enumerate}
\end{exm}

For every fs log scheme $X$,
let $\partial X$ be the closed subset of $X$ consisting of $x\in X$ such that $\ol{\cM}_{X,x}$ is nontrivial.
We regard $\partial X$ as the strict closed subscheme of $X$ with reduced scheme structure.

\begin{df}
\label{lognearby.1}
Let $X$ be a regular log regular fs log scheme in $\lSch/B$,
i.e.,
an fs log scheme in $\lSch/B$ such that $\ul{X}$ is regular and $X$ is log regular.
The \emph{log motivic nearby cycles functor for $X$} is 
\[
\Psi_X^{\log}:=i^*j_*
\colon \sT(X-\partial X)
\to
\sT(\partial X),
\]
where $i\colon \partial X\to X$ and  $j\colon X-\partial X\to X$ be the obvious strict immersions.
We often omit the subscript $X$ in $\Psi_X^{\log}$ if $X$ is clear from the context.
\end{df}

\begin{prop}
\label{lognearby.2}
Let $X$ be a regular log regular scheme in $\lSch/B$,
and let $j\colon X-\partial X\to X$ be the obvious open immersion.
Then the induced morphism
\[
\unit \xrightarrow{\ad} j_*j^*\unit
\]
is an isomorphism in the following cases:
\begin{enumerate}
\item[\textup{(1)}]
There exists a log smooth morphism $X\to S$ with $S\in \Sch/B$.
\item[\textup{(2)}]
$\sT=\Mod_\bE$ for some commutative algebra object $\bE$ of $\SH(B)$ satisfying absolute purity in the sense of \cite[Definition 3.5]{regGysin}.
\item[\textup{(3)}]
$B$ is the spectrum of a perfect field,
and $\sT=\Mod_\bE$ for some commutative algebra object $\bE$ of $\SH(B)$,
\item[\textup{(4)}]
$\sT(X)=\DA(X,\Q),\Mod_{\rM \Q}(X)$.
\end{enumerate}
\end{prop}
\begin{proof}
(1) is a consequence of ($ver$-inv) in \cite[Theorem 2.3.5]{logshriek}.
(2) is a consequence of \cite[Theorem 3.7]{regGysin}.
(3) is a consequence of (2) and \cite[Theorem C.1]{MR4225026}.
(4) is a consequence of (2), \cite[Theorem 3.6(2)]{MR4225026}, and \cite[Theorems 14.4.1, 16.1.4]{CD12}.
\end{proof}

\begin{thm}
\label{lognearby.3}
Let $f\colon V\to X$ be a proper vertical exact log smooth morphism in $\lSch/B$,
and let $g\colon V-\partial V\to X-\partial X$ and $h\colon V\times_X \partial X\to \partial X$ be the induced morphisms.
Assume that the induced morphism $\unit \xrightarrow{\ad} j_*j^*\unit$ is an isomorphism,
where $j\colon X-\partial X\to X$ is the obvious open immersion.
Then there are natural isomorphisms in $\sT^\ex(X)$
\[
j_*M_{X-\partial X}(V-\partial V)
\simeq
M_X(V),
\text{ }
j_*g_*\unit
\simeq
f_*\unit
\]
and natural isomorphisms in $\sT^\ex(\partial X)$
\[
\Psi^{\log}M_{X-\partial X}(V-\partial V)
\simeq
M_{\partial X}(V\times_X \partial X),
\text{ }
\Psi^{\log}g_*\unit
\simeq
h_*\unit.
\]
\end{thm}
\begin{proof}
By \cite[Proposition 2.3.9]{logA1},
the projection $V\times_X (X-\partial X)\to X-\partial X$ is vertical,
so $V\times_X (X-\partial X)$ has the trivial log structure.
This implies that the left square in the induced diagram
\[
\begin{tikzcd}
V-\partial V\ar[d,"g"']\ar[r,"j'"]&
V\ar[d,"f"]\ar[r,"i'",leftarrow]&
V\times_X \partial X\ar[d,"h"]
\\
X-\partial X\ar[r,"j"]&
X\ar[r,"i",leftarrow]&
\partial X
\end{tikzcd}
\]
is cartesian.
We have natural isomorphisms
\[
j_*g_\sharp \unit
\simeq^{(1)}
j_* g_* \Sigma_{g}^n \unit
\simeq
f_* j_*' \Sigma_{g}^n \unit
\simeq^{(2)}
f_* \Sigma_f^n j_*' \unit
\simeq^{(3)}
f_* \Sigma_f^n \unit
\simeq^{(4)}
f_\sharp \unit,
\]
where (1) and (4) are due to \cite[Theorem 1.3.1(3)]{logsix},
(2) is due \cite[Proposition 2.2.7]{logsix},
and (3) is due to the assumption.
We refer to \cite[Definition 2.4.1]{logsix} for the notation $\Sigma_f^n$.
We also have natural isomorphisms
\[
j_*g_*\unit
\simeq
f_*j_*'\unit
\simeq^{(5)}
f_*\unit,
\]
where (5) is due to the assumption.
Apply $i^*$ to the above isomorphisms and use ($\eSm$-BC) in \cite[Theorem 1.2.1]{logsix} and \cite[Theorem 1.3.1(1)]{logsix} to conclude.
\end{proof}

\begin{prop}
\label{lognearby.8}
Let $X$ be the spectrum of a DVR of equal characteristic $0$ with the standard log structure.
Then for every projective smooth morphism of schemes $U\to X-\partial X$,
there exists a projective log smooth morphism of schemes $V\to X$ such that $V-\partial V\simeq U$.
\end{prop}
\begin{proof}
By resolution of singularities,
there exists a projective regular scheme $\ul{V}$ over $\ul{X}$ such that $\ul{V}$ contains $U$ as a dense open subscheme and the complement of $U$ in $\ul{V}$ is a strict normal crossing divisor.
Let $V$ be the fs log scheme with the underlying scheme $\ul{V}$ and the compactifying log structure associated with $U$.
Then we have $V\in \lSm/X$ by \cite[Corollary IV.3.1.18]{Ogu}.
\end{proof}

\begin{rmk}
\label{lognearby.9}
Proposition \ref{lognearby.8} does not hold in general if $X$ is the spectrum of a DVR of equal characteristic $p>0$ or mixed characteristic.
We only give an example for the equal positive characteristic case.
Assume that $X$ is the spectrum of $k[\! [x] \! ]$ with the standard log structure,
where $k$ is a field of characteristic $p>0$.
Consider $U:=\Spec(k(\! (x )\! )[t]/(t^p-xt-1))$,
which is finite \'etale over $X-\partial X\simeq \Spec(k(\! (x )\! ))$.

If there exists a proper log smooth morphism $V\to X$ such that $V-\partial V\simeq U$,
then $V\to X$ is log \'etale.
Furthermore,
$\ul{V}$ is the spectrum of the integral closure of $k[\! [x] \! ]$ in $k(\! (x )\! )[t]/(t^p-xt-1)$,
which is $k[\! [x] \! ][t]/(t^p-xt-1)$.
Since $V$ is regular log regular,
the log structure on $V$ is given by $\N t\to k[\! [x] \! ][t]/(t^p-xt-1)$.
It follows that $V\times_X \partial X$ is the spectrum of the log ring $\N t\to k[t]/(t^p-1)$.
By computing $\Omega_{V\times_X \partial X/\partial X}^1$,
we see that the projection $V\times_X \partial X\to \partial X$ is not log \'etale,
which is a contradiction.
\end{rmk}

\begin{thm}
\label{lognearby.4}
Let $X\in \lSch/B$ be the spectrum of a DVR of equal characteristic $0$ with the standard log structure.
Then the functor $j_*\colon \sT(X-\partial X)\to \sT(X)$ is symmetric monoidal.
Hence the functor $\Psi^{\log}\colon \sT(X-\partial X) \to \sT(\partial X)$ is symmetric monoidal too.
\end{thm}
\begin{proof}
Since $j_*$ is a right adjoint of the symmetric monoidal functor $j^*$,
$j_*$ is lax monoidal.
Hence we have the induced natural morphism
\[
j_*\cF\otimes j_*\cG
\to
j_*(\cF\otimes \cG)
\]
for $\cF,\cG\in \sT(X-\partial X)$.
We need to show that this is an isomorphism.
Recall that by resolution of singularities,
$\sT(X-\partial X)$ is generated under colimits by $M_{X-\partial X}(U)(d)[n]$ for projective $U\in\Sm/(X-\partial X)$ and $d,n\in \N$.
Since $\otimes$ preserves colimits in each variable and $j_*$ preserves colimits,
we reduce to the case when $\cF=M_{X-\partial X}(U)$ and $\cG=M_{X-\partial X}(U')$ for some projective $U,U'\in \Sm/(X-\partial X)$.
By Proposition \ref{lognearby.8},
we may assume that there exist projective $V,V'\in \lSm/X$ such that $V\times_X (X-\partial X)\simeq U$ and $V'\times_X (X-\partial X)\simeq U'$.
To conclude,
observe that we have the isomorphisms
\begin{align*}
& j_*M_{X-\partial X}(U)\otimes j_*M_{X-\partial X}(U')
\simeq
M_{X}(V)\otimes M_X(V')
\\
\simeq &
M_X(V\times_X V')
\simeq
j_*M_{X-\partial X}(U\times_{X-\partial X} U'),
\end{align*}
where the first and third isomorphisms are due to Theorem \ref{lognearby.3}.
\end{proof}

Let us record two basic properties of the log motivic nearby cycles functors.

\begin{prop}
\label{lognearby.6}
Let $X$ be a regular log regular scheme in $\lSch/B$,
let $i\colon \partial X\to X$ and $j\colon X-\partial X\to X$ be the obvious strict immersions,
and let $r\colon \partial X\to \ul{\partial X}$ be the morphism removing the log structure.
Then there is a natural isomorphism
\[
\ul{i}^*\ul{j}_*
\simeq
r_*\Psi^{\log}.
\]
\end{prop}
\begin{proof}
Consider the commutative diagram with cartesian squares
\[
\begin{tikzcd}
X-\partial X\ar[d,"\id"']\ar[r,"j"]&
X\ar[d,"p"]\ar[r,"i",leftarrow]&
\partial X\ar[d,"r"]
\\
X-\partial X\ar[r,"\ul{j}"]&
\ul{X}\ar[r,leftarrow,"\ul{i}"]&
\ul{\partial X},
\end{tikzcd}
\]
where $p$ is the morphism removing the log structure.
We have the natural isomorphisms
\[
\ul{i}^*\ul{j}_*
\simeq
\ul{i}^*p_*j_*
\simeq
r_*i^*j_*,
\]
where the second isomorphism is due to \cite[Theorem 1.3.1(1)]{logsix}.
\end{proof}

\begin{prop}
\label{lognearby.7}
Let $f\colon Y\to X$ be a vertical exact log smooth morphism of regular log regular schemes in $\Sch/B$.
Then the diagram
\[
\begin{tikzcd}
\sT^\ex(X-\partial X)\ar[r,"\Psi_X^{\log}"]\ar[d,"g^*"']&
\sT^\ex(\partial X)\ar[d,"h^*"]
\\
\sT^\ex(Y-\partial Y)\ar[r,"\Psi_Y^{\log}"]&
\sT^\ex(\partial Y),
\end{tikzcd}
\]
commute,
where $g\colon Y-\partial Y\to X-\partial Y$ and $h\colon \partial X\to \partial Y$ are the induced morphisms.
\end{prop}
\begin{proof}
Since $f$ is vertical,
the induced square
\[
\begin{tikzcd}
Y-\partial Y\ar[d,"g"']\ar[r]&
Y\ar[d,"f"]
\\
X-\partial X\ar[r]&
X
\end{tikzcd}
\]
is cartesian.
Use ($\eSm$-BC) in \cite[Theorem 1.2.1]{logsix} to conclude.
\end{proof}

\section{Unipotent motives are strict motives over the standard log point}

Throughout this section,
we fix $S\in \Sch/B$.

For an fs monoid $P$ and an fs log scheme $X$,
we set
\[
G_{m,X}:=X\times \G_m,
\text{ }
\A_{P,X}:=X\times \A_P.
\]
For a sharp fs monoid $P$,
we set
\[
\pt_P:=\A_P\times_{\ul{\A_P}}\{O\},
\]
where $O$ is the ``origin'' of $\ul{\A_P}$.
This is the log point associated with $P$.
We also set
\[
\pt_{P,X}:=X\times \pt_P.
\]
Recall that for a field $k$,
$k^\dagger$ denotes the log ring $(k,\N\oplus k^*)$.
Note that we have $\Spec(k^\dagger)=\pt_{\N,\Spec(k)}$.

Recall the $\infty$-category of unipotent motives $\UT(S)$ whose formulation is due to Ayoub and Spitzweck.
This is the full subcategory of $\sT(\G_{m,S})$ generated under colimits by $M_{\G_{m,S}}(\G_{m,X})(d)[n]$ for $X\in \Sm/S$ and $d,n\in \Z$.
The purpose of this section is to compare $\sT(\pt_{\N,S})$ with $\UT(S)$,
see Theorem \ref{uni.1} below.

Recall the $\infty$-category of motives with monodromy $\DA_{N,\et}(k,\Q)$ for a perfect field $k$ introduced by Binda-Gallauer-Vezzani \cite[Definition 4.10]{2306.05099}.
By \cite[Corollary 2.61]{2306.05099},
there is an equivalence of $\infty$-categories
\begin{equation}
\label{uni.1.1}
\DA_{N,\et}(k,\Q)
\simeq
\rU\DA_{\et}(k,\Q).
\end{equation}

\begin{lem}
\label{uni.2}
Let $F\colon \cC\to \cD$ be a colimit preserving functor of stable $\infty$-categories admitting colimits.
Assume that $\cC$ is generated under colimits by a family $\cA$ of compact objects.
If $F(X)$ is compact for all $X\in \cA$,
then a right adjoint $G$ of $F$ preserves colimits.
If we further assume that the induced morphism
\[
\Hom_{\cC}(X,Y)
\to
\Hom_{\cD}(F(X),F(Y))
\]
is an isomorphism for all $X,Y\in \cA$,
then $F$ is fully faithful.
\end{lem}
\begin{proof}
It is well-known that $G$ preserves colimits,
which is a consequence of \cite[Proposition 1.4.4.1(2)]{HA} and the fact that a functor corepresented by a compact object preserves coproducts.

The assumption on $\Hom$ implies that the unit $Y\xrightarrow{\ad} GF(Y)$ is an isomorphism for $Y\in \cA$
since $\cA$ generates $\cC$ under colimits.
It follows that the unit $\id \xrightarrow{\ad} GF$ is an isomorphism since $\cF$ and $\cG$ preserve colimits and $\cA$ generates $\cC$ under colimits.
\end{proof}

\begin{prop}
\label{uni.8}
Let $j\colon \G_{m,S}\to \A_{\N,S}$ be the obvious open immersion.
Then the morphism
\[
M_{\A_{\N,S}}(\A_{\N,X})
\xrightarrow{\ad}
j_*j^*M_{\A_{\N,S}}(\A_{\N,X})
\]
is an isomorphism.
\end{prop}
\begin{proof}
This is a consequence of ($ver$-inv) in \cite[Theorem 1.2.1]{logsix}.
\end{proof}

\begin{thm}
\label{uni.1}
The functor $\Psi^{\log}\colon \sT(\G_{m,S})\to \sT(\pt_{\N,S})$ restricts to an equivalence of $\infty$-categories
\begin{equation}
\label{uni.1.2}
\Psi^{\log}
\colon
\UT(S)
\xrightarrow{\simeq}
\sT^\st(\pt_{\N,S}).
\end{equation}
In particular,
there is an equivalence of $\infty$-categories
\begin{equation}
\label{uni.1.3}
\Mod_{\rM \Q}^\st(k^\dagger)
\simeq
\DA_{N,\et}(k,\Q),
\end{equation}
where $k$ is a perfect field.
\end{thm}
\begin{proof}
Let $p\colon \G_{m,S}\to S$ and $q\colon \A_{\N,S}\to S$ be the projections,
and let $i\colon \pt_{\N,S}\to \A_{\N,S}$ and $j\colon \G_{m,S}\to \A_{\N,S}$ be the obvious strict immersions.
Note that \eqref{uni.1.2} preserves colimits.
By Lemma \ref{uni.2} and Proposition \ref{uni.8},
to show that \eqref{uni.1.2} is fully faithful,
it suffices to show that the induced composite morphism
\begin{align*}
&\Hom_{\sT(\G_{m,S})}(M_{\G_{m,S}}(\G_{m,X}),M_{\G_{m,S}}(\G_{m,Y})(d)[n])
\\
\to &
\Hom_{\sT(\A_{\N,S})}(M_{\A_{\N,S}}(\A_{\N,X}),M_{\A_{\N,S}}(\A_{\N,Y})(d)[n])
\\
\to &
\Hom_{\sT(\pt_{\N,S})}(M_{\pt_{\N,S}}(\pt_{\N,X}),M_{\pt_{\N,S}}(\pt_{\N,Y})(d)[n])
\end{align*}
is an isomorphism for $X,Y\in \Sm/S$ and $d,n\in \Z$.
The first morphism is an isomorphism by Proposition \ref{uni.8} since $j_*$ is fully faithful,
so it suffices to show that the second morphism is an isomorphism.

For this,
by adjunction,
it suffices to show that the natural transformation
\[
q_*q^*\xrightarrow{\ad}
q_*i_*i^*q^*
\]
is an isomorphism.
By (Loc) in \cite[Theorem 1.2.1]{logsix},
it suffices to show
\[
p_!p^*\simeq 0.
\]
By (Loc) in \cite[Theorem 1.2.1]{logsix} again,
it suffices to show that the natural transformation
\[
\ul{q}_*\ul{q}^*\xrightarrow{\ad}
\ul{q}_*\ul{i}_*\ul{i}^*\ul{q}^*
\]
is an isomorphism,
where $\ul{q}\colon \A_S^1\to S$ is the projection, and $\ul{i}\colon S\to \A_S^1$ is the $0$-section.
This is a consequence of $\A^1$-invariance.
Hence \eqref{uni.1.2} is fully faithful.
To show that it is essentially surjective,
observe that it preserves colimits and its essential image generates $\sT^\st(\pt_{\N,S})$ under colimits.

The equivalence \eqref{uni.1.3} is due to Example \ref{lognearby.5}(3) and \eqref{uni.1.1}.
\end{proof}

\begin{df}
\label{uni.4}
Let $o^*$ be the composite of the natural transformations
\[
\sT^\st(\pt_{\N,S})
\xrightarrow{(\Psi^{\log})^{-1}} \UT(S)
\xrightarrow{i_1^*}
\sT(S),
\]
where $i_1\colon S\to \G_{m,S}$ is the $1$-section.
The notation $o$ stands for the ``origin'' of $S^1$.
Even though there exist no morphisms of fs log schemes $o\colon S\to \pt_{\N,S}$,
the functor $o^*$ behaves like a pullback functor.
\end{df}

We will extend the functor $o^*\colon \sT^{\st}(\pt_{\N,S})\to \sT(S)$ to $o^*\colon \sT(\pt_{\N,S})\to \sT(S)$ in Theorem \ref{nearby.4}.

\begin{prop}
\label{uni.5}
The functor
\[
j_*
\colon
\UT(S)
\to
\sT(\A_{\N,S})
\]
is symmetric monoidal,
where $j\colon \G_{m,S}\to \A_{\N,S}$ is the obvious open immersion.
Hence the functors
\[
\Psi^{\log} \colon \UT(S)
\to
\sT^\st(\pt_{\N,S}),
\text{ }
o^*\colon \sT^\st(\pt_{\N,S})
\to
\sT(S)
\]
are symmetric monoidal.
\end{prop}
\begin{proof}
Since $j_*$ is lax monoidal and preserves colimits,
to show that $j_*
\colon
\UT(S)
\to
\sT(\A_{\N,S})$ is monoidal,
it suffices to show that the natural morphism in $\sT(\A_{\N,S})$
\[
j_*M_{\G_{m,S}}(\G_{m,X})
\otimes
j_*M_{\G_{m,S}}(\G_{m,Y})
\to
j_*(M_{\G_{m,S}}(\G_{m,X})\otimes M_{\G_{m,S}}(\G_{m,X}))
\]
is an isomorphism for $X,Y\in \Sm/S$.
This is a consequence of Proposition \ref{uni.8}.

For the other claims,
use the facts that the functors $i^*\colon \sT(\A_{\N,S})\to \sT(\pt_{\N,S})$ and $i_1^*\colon \sT(\G_{m,S})\to \sT(S)$ are symmetric monoidal,
where $i\colon \pt_{\N,S}\to \A_{\N,S}$ is the obvious strict closed immersion,
and $i_1$ is the $1$-section.
\end{proof}

\begin{prop}
\label{uni.6}
The composite functor
\[
\sT(S)\xrightarrow{p^*} \sT(\G_{m,S})
\xrightarrow{\Psi^{\log}}
\sT(\pt_{\N,S})
\]
is naturally isomorphic to $r^*$,
where $p\colon \G_{m,S}\to S$ and $r\colon \pt_{\N,S}\to S$ are the projections.
\end{prop}
\begin{proof}
This is an immediate consequence of Proposition \ref{uni.8}.
\end{proof}

\begin{prop}
\label{uni.7}
The composite functor
\[
\sT(S)\xrightarrow{r^*}
\sT^\st(\pt_{\N,S})\xrightarrow{o^*}
\sT(S)
\]
is naturally isomorphic to $\id$,
where $r\colon \pt_{\N,S}\to S$ is the projection.
\end{prop}
\begin{proof}
We have $i_1^*p^*\simeq \id$,
where $i_1^*\colon S\to \G_{m,S}$ is the $1$-section,
and $p\colon \G_{m,S}\to S$ is the projection.
Proposition \ref{uni.6} finishes the proof.
\end{proof}

\section{From motives over the standard log point to motives over the point}

Throughout this section,
we fix $S\in \Sch/B$.
For $r\in \N^+$,
let
\[
w_r\colon \A_{\N,S}\to \A_{\N,S},
\text{ }
w_r\colon \G_{m,S}\to \G_{m,S},
\text{ }
w_r\colon \pt_{\N,S}\to \pt_{\N,S}
\]
be the morphisms of fs log schemes induced by the multiplication $r\colon \N\to \N$.

The purpose of this section is to construct a suitable functor
\[
o^*\colon \sT(\pt_{\N,S})\to \sT(S)
\]
extending the functor $o^*\colon \sT^{\st}(\pt_{\N,S})\to \sT(S)$ in Definition \ref{uni.4} and study its properties.

\begin{df}
\label{nearby.1}
For $r\in \N^+$,
let $\sT^{\sat_r}(\pt_{\N,S})$ be the full subcategory of $\sT(\pt_{\N,S})$ generated under colimits by $M_{\pt_{\N,S}}(X)(d)[n]$ for all $X\in \lSm/\pt_{\N,S}$ and $d,n\in \Z$ such that the projection $X\times_{\pt_{\N,S},w_r}\pt_{\N,S}\to X$ is saturated.

Note that we have $\sT^{\sat_1}(\pt_{\N,S})=\sT^{\st}(\pt_{\N,S})$ by \cite[Proposition 2.5.3]{logshriek}.
\end{df}

\begin{prop}
\label{nearby.2}
For $X\in \lSm/\pt_{\N,S}$,
there exists $r\in \N^+$ such that the projection $X\times_{\pt_{\N,S},w_r}\pt_{\N,S}\to X$ is saturated.
\end{prop}
\begin{proof}
The question is Zariski local on $X$,
so we may assume that there is a chart $\N\to P$ of $f\colon X\to \lSm/\pt_{\N,S}$.
By \cite[Theorem I.4.9.1]{Ogu},
there exists $r\in \N^+$ such that the pushout of $\N\to P$ along $r\colon \N\to \N$ is saturated.
Then the projection $X\times_{\pt_{\N,S},w_r}\pt_{\N,S}\to X$ is saturated.
\end{proof}

\begin{prop}
\label{nearby.3}
The triangle
\[
\begin{tikzcd}
\sT^\st(\pt_{\N,S})\ar[r,"o^*"]\ar[d,"w_r^*"']&
\sT(S)
\\
\sT^\st(\pt_{\N,S})\ar[ru,"o^*"']
\end{tikzcd}
\]
commutes for every $r\in \N^+$.
\end{prop}
\begin{proof}
It suffices to show that the middle square in the diagram
\[
\begin{tikzcd}
\sT(\pt_{\N,S})\ar[r,leftarrow,"i^*"]\ar[d,"w_r^*"']&
\sT(\A_{\N,S})\ar[r,leftarrow,"j_*"]\ar[d,"w_r^*"]&
\UT(\G_{m,S})\ar[d,"w_r^*"]\ar[r,"i_1^*"]&
\UT(S)
\\
\sT(\pt_{\N,S})\ar[r,leftarrow,"i^*"]&
\sT(\A_{\N,S})\ar[r,leftarrow,"j_*"]&
\UT(\G_{m,S})\ar[ur,"i_1^*"']
\end{tikzcd}
\]
commutes,
where $i\colon \pt_{\N,S}\to \A_{\N,S}$ and $j\colon \G_{m,S}\to \A_{\N,S}$ be the obvious strict immersions.
Let $X\to S$ be a strict proper morphism in $\Sch/B$.
Consider the induced cartesian square
\[
\begin{tikzcd}
\G_{m,X}\ar[d,"j'"']\ar[r,"g"]&
\G_{m,S}\ar[d,"j"]
\\
\A_{\N,X}\ar[r,"f"]&
\A_{\N,S}.
\end{tikzcd}
\]
By \cite[Proposition 4.2.13]{CD12},
it suffices to show that the natural transformation
\[
w_r^*j_*f_*\unit
\xrightarrow{\Ex}
j_*w_r^*f_*\unit
\]
is an isomorphism,
where the notation $\Ex$ means that the natural transformation is induced by a Beck-Chevalley transformation.
Using \cite[Theorem 1.3.1(1)]{logsix},
it suffices to show that the natural transformation
\[
f_*w_r'^*j'_*\unit
\xrightarrow{\Ex}
f_*j_*'w_r'^*\unit
\]
is an isomorphism,
where $w_r'\colon \A_{\N,X}\to \A_{\N,X}$ is the pullback of $w_r$.
This is a consequence of $j_*'\unit\simeq \unit$ shown in Proposition \ref{lognearby.2}(1).
\end{proof}

Let $(\N^+,\times)$ be the monoid $\N^+$ with the multiplicative operation.
We regard this as a category.
Explicitly, we have
\[
\Hom_{(\N^+,\times)}(r,s)
:=
\left\{
\begin{array}{ll}
* & \text{if $r\vert s$},
\\
\emptyset & \text{otherwise.}
\end{array}
\right.
\]

\begin{prop}
\label{nearby.15}
The functor of $\infty$-categories
\[
\iota^*\colon
\colimPrL_{r\in (\N^+,\times)}
\sT^{\sat_r}(\pt_{\N,S})
\to
\sT(\pt_{\N,S})
\]
induced by the inclusions $\sT^{\sat_r}(\pt_{\N,S})\to \sT(\pt_{\N,S})$ is an equivalence,
where the colimit is taken in $\PrL$.
\end{prop}
\begin{proof}
We claim that $\iota^*$ is an equivalence of $\infty$-categories.
For $r\in (\N^+,\times)$,
the inclusion functor $\iota_r^* \colon \sT^{\sat_r}(\pt_{\N,S})\to \sT(\pt_{\N,S})$ admits a right adjoint $\iota_{r*}$.
Since the essential image of $\iota^*$ contains $M_{\pt_{\N,S}}(X)(d)[n]$ for all $X\in \lSm/S$ and $d,n\in \N$ by Proposition \ref{nearby.2},
it suffices to show that $\iota^*$ is fully faithful.
Under the identifications
\[
\colimPrL_{r\in (\N^+,\times)}
\sT^{\sat_r}(\pt_{\N,S})
\simeq
\limPrR_{r\in (\N^+,\times)^\op}
\sT^{\sat_r}(\pt_{\N,S})
\simeq
\lim_{r\in (\N^+,\times)^\op}
\sT^{\sat_r}(\pt_{\N,S})
\]
due to \cite[Corollary 5.5.3.4, Proposition 5.5.7.6]{HTT},
the functor
\[
\iota_*
\colon
\sT(\pt_{\N,S})
\to
\lim_{r\in (\N^+,\times)^\op}
\sT^{\sat_r}(\pt_{\N,S})
\]
induced by $\iota_{r*}$ for all $r\in (\N^+,\times)$ is a right adjoint of $\iota^*$.
By Lemma \ref{uni.2},
$\iota_*$ preserves colimits.
Hence to show that $\iota^*$ is fully faithful,
it suffices to show that the morphism $M_{\pt_{\N,S}}(X)\xrightarrow{\ad} \iota_*\iota^* M_{\pt_{\N,S}}(X)$ is an isomorphism for every $X\in \lSm/S$ such that the projection $X\times_{\pt_{\N,S},w_r}\pt_{\N,S}\to X$ is saturated for some $r\in (\N^+,\times)$,
where we regard $M_{\pt_{\N,S}}(X)$ as an object of $\sT^{\sat_r}(\pt_{\N,S})$.
This holds since the morphism $M_{\pt_{\N,S}}(X)\xrightarrow{\ad} \iota_{s*}\iota_s^* M_{\pt_{\N,S}}(X)$ is an isomorphism whenever $r\vert s$.
\end{proof}

\begin{thm}
\label{nearby.4}
There exists a colimit preserving symmetric monoidal functor of symmetric monoidal $\infty$-categories
\begin{equation}
\label{nearby.4.2}
o^*\colon \sT(\pt_{\N,S})
\to
\sT(S)
\end{equation}
satisfying the following two properties:
\begin{enumerate}
\item[\textup{(1)}]
$o^*\pi^*\simeq \id$,
where $\pi\colon \pt_{\N,S}\to S$ is the projection.
\item[\textup{(2)}]
$o^*w_r^*\simeq o^*$ for every integer $r\in \N^+$.
\end{enumerate}
\end{thm}
\begin{proof}
For $r\in \N^+$,
Consider the restriction
\[
w_r^*
\colon
\sT^{\sat_r}(\pt_{\N,S})
\to
\sT^{\st}(\pt_{\N,S})
\]
of the functor
\(
w_r^*
\colon
\sT(\pt_{\N,S})
\to
\sT(\pt_{\N,S}).
\)
Consider also the composite functor
\[
o^*
\colon
\sT^{\sat_r}(\pt_{\N,S})
\xrightarrow{w_r^*}
\sT^{\st}(\pt_{\N,S})
\xrightarrow{o^*}
\sT(S),
\]
which is symmetric monoidal by Proposition \ref{uni.5}.
The triangle
\[
\begin{tikzcd}
\sT^{\sat_r}(\pt_{\N,S})\ar[r,"o^*"]\ar[d,hookrightarrow]&
\sT(S)
\\
\sT^{\sat_s}(\pt_{\N,S})\ar[ru,"o^*"']
\end{tikzcd}
\]
commutes by Proposition \ref{nearby.3} for every $s\in \N^+$ such that $r\vert s$,
where the left vertical arrow is the inclusion.
From this,
we obtain a functor
\begin{equation}
\label{nearby.4.1}
o^*
\colon
{\colimPrL_{r\in (\N^+,\times)}}
\sT^{\sat_r}(\pt_{\N,S})
\to
\sT(S).
\end{equation}
By \cite[Corollary 3.2.3.2]{HA},
the source of \eqref{nearby.4.1} equips a natural symmetric monoidal structure,
and \eqref{nearby.4.1} is symmetric monoidal.
Together with Proposition \ref{nearby.15},
we obtain \eqref{nearby.4.2}.
This satisfies the property (1) by Proposition \ref{uni.7} and (2) by construction.
\end{proof}

For $X\in \lSch/B$ and $\bE\in \sT(B)$,
we set $\bE:=p^*\bE\in \sT(X)$ for simplicity of notation,
where $p\colon X\to B$ is a structure morphism.

\begin{prop}
\label{nearby.8}
For $\bE\in \sT(B)$,
we have a natural isomorphism in $\sT(S)$
\[
o^*\bE
\simeq
\bE.
\]
\end{prop}
\begin{proof}
This is an immediate consequence of Theorem \ref{nearby.4}(1).
\end{proof}

\begin{df}
\label{nearby.14}
For regular log regular fs log scheme $X$ in $\lSch/B$ such that $\partial X\simeq \pt_{\N,S}$,
the \emph{nearby cycles functor} is the composite functor
\[
\Psi_X
\colon
\sT(X-\partial X)
\xrightarrow{\Psi_X^{\log}}
\sT(\partial X)
\xrightarrow{o^*}
\sT(S).
\]
We often omit the subscript $X$ in $\Psi_X$ when it is clear from the context.
\end{df}

\begin{thm}
\label{nearby.9}
For $\bE\in \SH(B)$ satisfying the logarithmic absolute purity and regular log regular $X\in \lSch/B$ such that $\partial X\simeq \pt_{\N,S}$,
we have a natural isomorphism in $\SH(S)$
\[
\Psi (\bE)\simeq \bE.
\]
In particular,
we have a natural isomorphism in $\SH(S)$
\[
\Psi(\KGL)\simeq\KGL.
\]
\end{thm}
\begin{proof}
This is an immediate consequence of \cite[Example 3.6, Theorem 3.7]{regGysin} and Proposition \ref{nearby.8}.
\end{proof}

Hence if $X$ is the spectrum of $\Z_p$ with the standard log structure,
then the functor $\Psi$ creates the homotopy $K$-theory spectra of smooth schemes over $\F_p$ from the homotopy $K$-theory spectra of smooth schemes over $\Q_p$.

\section{Description of the functor \texorpdfstring{$o^*$}{o*}}
\label{description}

Throughout this section,
we fix $S\in \Sch/B$.
In this section,
we provide the description of the functor $o^*$,
see Theorem \ref{description.3} below.

For effective Cartier divisors $D_1,\ldots,D_n$ on a scheme $X$ given by invertible sheaves of ideals $\cI_1,\ldots,\cI_n$,
let $(X,D_1+\cdots+D_n)$ be the fs log scheme whose underlying scheme is $X$ with the Deligne-Faltings structure \cite[\S III.1.7]{Ogu} associated with the inclusions $\cI_1,\ldots,\cI_n\to \cO_X$.

We first introduce the notion of $\pt_P$-torsors as follows for technical convenience.

\begin{df}
\label{nearby.10}
For a sharp fs monoid $P$ and a scheme $X$,
a \emph{$\pt_P$-torsor over $X$} is an fs log scheme $V$ over $X$ such that $\ul{V}\to X$ is an isomorphism and $\ol{\cM}_{V,v}\simeq P$ for every point $v\in V$.

We say that $V$ is a \emph{trivial $\pt_P$-torsor over $X$} if there is an isomorphism $V\simeq \pt_{P,X}$.
\end{df}

\begin{prop}
\label{nearby.11}
Let $X$ be a scheme,
and let $P$ be a sharp fs monoid.
Then a $\pt_P$-torsor $V$ over $X$ is Zariski locally a trivial $\pt_P$-torsor.
\end{prop}
\begin{proof}
For a point $v$ of $V$,
there is a Zariski neighborhood $U$ of $v$ in $V$ such that $U$ has a neat chart $P$ at $v$ by \cite[Proposition II.2.3.7]{Ogu}.
Since $\ol{\cM}_{U,u}\simeq P$ for every point $u$ of $U$,
the induced morphism $U\to X\times \A_P$ factors through $U\to X\times \pt_P$,
which is an isomorphism.
\end{proof}

\begin{prop}
\label{nearby.12}
Let $X$ be a scheme.
Then the functor
\begin{align*}
\alpha\colon &\textup{(The subcategory of line bundles over $X$ consisting of isomorphisms)}
\\
\rightarrow &
\textup{(The subcategory of $\pt_\N$-torsors over $X$ consisting of isomorphisms)}
\end{align*}
given by $\cL\mapsto (\cL,Z)\times_\cL Z$ is an equivalence of categories,
where $Z$ is the $0$-section of $\cL$.
\end{prop}
\begin{proof}
Let $\cL$ be a line bundle over $X$,
and let $\{X_i\}_{i\in I}$ be a Zariski covering of $X$ such that there are isomorphisms $\varphi_i\colon \A^1_{X_i}\xrightarrow{\simeq} \cL_i:=\cL\times_X X_i$ for $i\in I$.
For $i,j\in I$,
we set $X_{ij}:=X_i\times_X X_j$ and $\cL_{ij}:=\cL_i\times_\cL \cL_j$.
We have the composite morphism
\begin{equation}
\label{nearby.12.1}
\A^1_{X_{ij}}
\xrightarrow{\varphi_j}
\cL_{ij}
\xrightarrow{\varphi_i^{-1}}
\A^1_{X_{ij}},
\end{equation}
which corresponds to an element $g_{ij}\in \Gamma(X_{ij},\cO_{ij}^*)$.

On the other hand,
let $V$ be a $\pt_\N$-torsor over $X$,
and let $\{X_i\}_{i\in I}$ be a Zariski covering of $X$ such that there are isomorphisms $\psi_i\colon \pt_{\N,X_i}\xrightarrow{\simeq} V_i:=V\times_X X_i$ for $i\in I$.
For $i,j\in I$,
we set $V_{ij}:=V_i\times_V V_j$.
We have the composite morphism
\begin{equation}
\label{nearby.12.2}
\pt_{\N,X_{ij}}
\xrightarrow{\psi_j}
V_{ij}
\xrightarrow{\psi_i^{-1}}
\pt_{\N,X_{ij}}.
\end{equation}
By taking $\Gamma(-,\cM_{-})$,
we have the induced composite map
\[
\N\oplus \Gamma(X_{ij},\cO_{X_{ij}}^*)
\xrightarrow{\simeq}
\Gamma(V_{ij},\cM_{V_{ij}})
\xrightarrow{\simeq}
\N\oplus \Gamma(X_{ij},\cO_{X_{ij}}^*).
\]
This sends $(1,0)$ to $(1,h_{ij})$ for some $h_{ij}\in \Gamma(X_{ij},\cO_X^*)$.

Now,
if $\alpha(\cL)=V$,
then \eqref{nearby.12.2} is induced by \eqref{nearby.12.1}.
This implies that we have $g_{ij}=h_{ij}$.
Using this,
one can readily show that $\alpha$ is an equivalence of categories.
\end{proof}

\begin{prop}
\label{nearby.13}
Let $Z\to X$ be a codimension $1$ regular immersion.
Then the $\pt_\N$-torsor $V:=Z\times_X (X,Z)$ over $Z$ corresponds to the normal bundle of $Z$ in $X$ under the equivalence in \textup{Proposition \ref{nearby.12}}.
\end{prop}
\begin{proof}
Let $\cI$ be the sheaf of ideals defining $Z$,
and let $\{X_i\}_{i\in I}$ be a Zariski covering of $X$ such that there is an isomorphism $\eta_i\colon \cO_{X_i}\xrightarrow{\simeq} \cI|_{X_i}$.
We set $X_{ij}:=X_i\times_X X_j$, $Z_i:=Z\times_X X_i$, $Z_{ij}:=Z_i\times_Z Z_j$, $V_i:=V\times_X X_i$, and $V_{ij}:=V_i\times_V V_j$ for $i,j\in I$.
Using the description of the log structure associated with a Deligne-Faltings structure in \cite[\S III.1.7]{Ogu},
$\eta_i(1)\in \Gamma(X_i,\cI|_{X_i})$ induces an isomorphism $\psi_i\colon \pt_{\N,Z_i}\xrightarrow{\simeq} V_i$.
Furthermore,
the induced composite map
\[
\N\oplus \Gamma(Z_{ij},\cO_{Z_{ij}}^*)
\xrightarrow{\simeq}
\Gamma(V_{ij},\cM_{V_{ij}})
\xrightarrow{\simeq}
\N\oplus \Gamma(Z_{ij},\cO_{Z_{ij}}^*)
\]
sends $(1,0)$ to $(1,h_{ij})$ for some $h_{ij}\in \Gamma(Z_{ij},\cO_X^*)$ such that $h_{ij}\eta_i(1)=\eta_j(1)$ in $\Gamma(Z_{ij},(\cI/\cI^2)|_{Z_{ij}})$.
To conclude,
observe that the collection of $h_{ij}$ for $i,j\in I$ corresponds to the vector bundle associated with the invertible sheaf $\cI/\cI^2$ on $Z$,
which is the normal bundle of $Z$ in $X$.
\end{proof}

We initiate the proof of Theorem \ref{description.3}.

\begin{lem}
\label{description.2}
For $X\in \lSch/B$ and $n\in \N^+$,
let $f\colon \A_{\N^n,X}\to X$ be the projection,
and let $i\colon \pt_{\N^n,X}\to \A_{\N^n,X}$ is the obvious strict closed immersion.
Then the natural transformation
\[
f_*f^*\xrightarrow{\ad} f_*i_*i^*f^*
\]
is an isomorphism.
\end{lem}
\begin{proof}
We proceed by induction on $n$.
If $n=1$,
then consider the induced cartesian square
\[
\begin{tikzcd}
\pt_{\N,X}\ar[r]\ar[d]&
\A_{\N,X}\ar[d]
\\
X\ar[r,"a"]&
\A^1_X,
\end{tikzcd}
\]
where $a$ is the $0$-section.
By \cite[Corollary 3.6.8]{logshriek},
it suffices to show that the natural transformation $g_*g^*\xrightarrow{\ad}g_*a_*a^*g^*$ is an isomorphism,
where $g\colon \A_X^1\to X$ is the projection.
This is a consequence of $\A^1$-invariance.

If $n>1$,
then consider the obvious strict closed immersions $\pt_{\N^n,X} \xrightarrow{c}\A_\N\times \pt_{\N^{n-1},X} \xrightarrow{b} \A_{\N^n,X}$.
By induction,
the natural transformations
\[
f_*f^*
\xrightarrow{\ad}
f_*b_*b^*f^*
\xrightarrow{\ad}
f_*b_*c_*c^*b^*f^*
\]
are isomorphisms.
Hence the composite natural transformation $f_*f^*\xrightarrow{\ad}f_*i_*i^*f^*$ is an isomorphism.
\end{proof}

\begin{const}
\label{description.6}
Let $X$ be a regular log regular fs log scheme in $\lSch/B$ such that $\partial X\simeq \pt_{\N,S}$,
and let $V\to X$ be a saturated vertical log smooth morphism in $\lSch/B$ such that $\ul{V}$ is regular.
Then the effective Cartier divisor $\ul{\partial V}$ is equal to the sum of effective Cartier divisors $\ul{D_1}+\cdots+\ul{D_n}$ for some $\ul{D_1},\ldots,\ul{D_n}\in \Sm/S$ such that each $\ul{D_i}$ is connected.
We set $D_i:=\ul{D_i}\times_{\ul{V}}(\ul{V},\ul{D_i})$ for each $i$.
Then $D_i$ is a $\pt_\N$-torsor over $\ul{D_i}$.
Let $\ul{\cE_i}$ be the corresponding line bundle over $\ul{D_i}$ under the equivalence in Proposition \ref{nearby.12}.
We set $\cE_i:=(\ul{\cE_i},\ul{D_i})$,
where $\ul{D_i}$ is regarded as the $0$-section.

For every nonempty subset $I$ of $\{1,\ldots,n\}$,
we set
\[
\ul{D_I}
:=
\bigcap_{i\in I} \ul{D_i},
\text{ }
\ul{D_I^\circ}
:=
\ul{D_I}-\bigcup_{j\notin I} \ul{D_j},
\text{ }
D_I'
:=
\ul{D_I}\times_{\ul{\partial V}}\partial V,
\text{ }
D_I^\circ
:=
\ul{D_I^\circ}\times_{\ul{\partial V}}\partial V.
\]
Observe that $D_I^\circ$ is a $\pt_{\N^{\lvert I \rvert}}$-torsor over $\ul{D_I^\circ}$,
which is the fiber product of the $\pt_\N$-torsors $D_i\times_{\ul{D_i}}\ul{D_I^\circ}$ for $i\in I$ over $\ul{D_I^\circ}$.
Let $\cE_I^\circ$ be the fiber product of $\cE_i\times_{\ul{D_i}}\ul{D_I^\circ}$ for $i\in I$ over $\ul{D_I^\circ}$.
Observe that $\ul{\cE_I^\circ}$ is the normal bundle of $\ul{D_I^\circ}$ in $V-\bigcup_{j\notin I}\ul{D_j}$ by Proposition \ref{nearby.13} and \cite[B.7.4]{Fulton}.

The induced morphism from the $\pt_{\N^{\lvert I \rvert}}$-torsor $D_I^\circ$ over $\ul{D_I^\circ}$ to the $\pt_\N$-torsor $\pt_{\N,S}$ over $S$ induces a morphism $\cE_I^\circ \to \A_{\N,S}$.
We set $F_I^\circ:=\cE_I^\circ\times_{\A_{\N,S}}\{1\}$ and $W_I^\circ:=\cE_I^\circ\times_{\A_{\N,S}}\pt_{\N,S}$.
Observe that $F_I^\circ$ is a $\G_m^{\lvert I \rvert-1}$-torsor over $\ul{D_I^\circ}$.
\end{const}

\begin{const}
\label{description.7}
We keep using the notation in Construction \ref{description.6}.
Zariski locally on $X$,
$V\to X$ admits a neat chart $\theta\colon \N\to Q:=\N^n$ such that the induced morphism $V\to X\times_{\A_\N}\A_Q$ is strict smooth by \cite[Proposition A.4]{divspc} and \cite[Remark 3.3.5]{logGysin}.
Let us assume that such a neat chart exists globally.
Then we can express various fs log schemes in Construction \ref{description.6} as follows.
For $1\leq i\leq n$,
we have
\[
D_i\simeq \ul{V}\times_{\ul{\A_Q}}(\A^{i-1}\times \pt_\N \times \A^{n-i-1}),
\text{ }
\cE_i\simeq \ul{D_i} \times \A_\N.
\]
If $I=\{1,\ldots,r\}$ with $1\leq r\leq n$,
then we have
\begin{gather*}
D_I'
\simeq
V\times_{\A_Q}(\pt_{\N^r} \times \A_{\N^{n-r}}),
\text{ }
D_I^\circ
\simeq
V\times_{\A_Q}(\pt_{\N^r} \times \G_m^{n-r})
\simeq
\ul{D_I^\circ}\times \pt_{\N^r},
\\
\cE_I^\circ
\simeq
\ul{D_I^\circ}\times \A_{\N^r},
\text{ }
F_I^\circ\simeq \ul{D_I^\circ}\times (\A_{\N^r}\times_{m,\A_\N} \{1\}),
\text{ }
W_I^\circ\simeq \ul{D_I^\circ}\times \partial \A_{\N^r},
\end{gather*}
where $m\colon \A_{\N^r}\to \A_\N$ is the morphism induced by the diagonal map $\N\to \N^r$.

Consider the factorization $\N\to P:=\N^n\to Q$ of $\theta$,
where the first map is the first inclusion,
and the second map is given by $(a_1,\ldots,a_n)\mapsto (a_1,a_1+a_2,\ldots,a_1+a_n)$.
The underlying morphism of schemes $\ul{\A_Q}\to \ul{\A_P}$ is given by $\Z[x_1,\ldots,x_n]\to \Z[x_1,\ldots,x_n]$ sending $x_1$ to $x_1\cdots x_n$ and $x_i$ to $x_i$ for $2\leq i\leq n$.
Hence we have the isomorphism from the closed subscheme $\Spec(\Z[x_1,\ldots,x_n]/(x_1,\ldots,x_r))$ of $\ul{\A_Q}$ to the closed subscheme $\Spec(\Z[x_1,\ldots,x_n]/(x_1,\ldots,x_r))$ of $\ul{\A_P}$.
By considering the log structures too,
we have the virtual isomorphism $u$ in the sense of \cite[Definition 3.2.1]{logsix} from the strict closed subscheme $\pt_{\N^r}\times \A_{\N^{n-r}}$ of $\A_Q$ to the strict closed subscheme $\pt_{\N^r}\times \A_{\N^{n-r}}$ of $\A_P$.
There is a unique commutative diagram with cartesian squares
\[
\begin{tikzcd}
D_I^\circ\ar[r]\ar[d]&
D_I' \ar[d]\ar[r]&
\pt_{\N^r}\times \A_{\N^{n-r}}\ar[d,"u"]
\\
Z_I^\circ\ar[r]&
Z_I'\ar[r]&
\pt_{\N^r} \times \A_{\N^{n-r}}.
\end{tikzcd}
\]
Note that the morphisms $D_I'\to Z_I'$ and $D_I^\circ\to Z_I^\circ$ are virtual isomorphisms.
Since the induced morphism $D_I'\to X\times_{\A_\N}(\pt_{\N^r}\times \A_{\N^r})$ is strict smooth,
the induced morphism $Z_I'\to X\times_{\A_\N}(\pt_{\N^r}\times \A_{\N^r})\simeq \partial X\times \pt_{\N^{r-1}} \times \A_{\N^r}$ is strict smooth too.
Hence we have $Z_I'\in \eSm/(\partial X\times \pt_\N^{r-1})$ and $Z_I^\circ\simeq Z_I'-\partial_{\partial X\times \pt_\N^{r-1}}Z_I'$,
see \cite[Definition 2.3.5]{logA1} for the notation $\partial_{\partial X\times \pt_\N^{r-1}}Z_I'$.
\end{const}

\begin{lem}
\label{description.8}
Under the notation of \textup{Construction \ref{description.6}},
the natural transformation
\[
f_{I*}f_I^*\xrightarrow{\ad} g_{I*}g_I^*
\]
is an isomorphism,
where $f_I\colon D_I'\to \partial X$ and $g_I\colon D_I^\circ \to \partial X$ are the induced morphisms.
\end{lem}
\begin{proof}
The question is Zariski local on $X$,
so we reduce to the case of Construction \ref{description.7}.
Let $a_I\colon Z_I'\to \partial X$ and $b_I\colon Z_I'\to \partial X$ be the induced morphisms.
By \cite[Theorem 3.2.6]{logsix},
the natural transformations $a_{I*}a_I^*\xrightarrow{\ad} f_{I*}f_I^*$ and $b_{I*}b_I^*\xrightarrow{\ad} g_{I*}g_I^*$ are isomorphisms.
To conclude,
observe that the natural transformation $a_{I*}a_I^*\xrightarrow{\ad}b_{I*}b_I^*$ is an isomorphism by ($ver$-inv) in \cite[Theorem 1.2.1]{logsix}.
\end{proof}

\begin{lem}
\label{description.9}
Under the notation of \textup{Construction \ref{description.6}},
the natural transformation
\[
h_{I*}h_I^*\xrightarrow{\ad} g_{I*}g_I^*
\]
is an isomorphism,
where $g_I\colon D_I^\circ\to \partial X$ and $h_I\colon W_I^\circ \to \partial X$ are the induced morphisms.
\end{lem}
\begin{proof}
The question is Zariski local on $X$,
so we reduce to the case of Construction \ref{description.7}.
Consider the face $F:=0\oplus \N^{n-1}$ of $P$ and $G:=0\oplus \N^{n-1}$ of $Q$.
We have the induced commutative diagram with cartesian squares
\[
\begin{tikzcd}
\pt_{Q,\ul{D_I^\circ}}\ar[r,"i'"]\ar[d]&
\A_{Q,\ul{D_I^\circ}}\times_{\A_{\N}}\pt_{\N}\ar[d]\ar[r,leftarrow,"j'"]&
\A_{Q_G,\ul{D_I^\circ}}\times_{\A_{\N}}\pt_{\N}\ar[d,"\simeq"]
\\
\pt_{P,\ul{D_I^\circ}}\ar[r,"i"]&
\A_{P,\ul{D_I^\circ}}\times_{\A_{\N}}\pt_{\N}\ar[r,leftarrow,"j"]&
\A_{P_F,\ul{D_I^\circ}}\times_{\A_{\N}}\pt_{\N}.
\end{tikzcd}
\]
Observe that the induced map $P_F\to Q_G$ is an isomorphism.
Let $b\colon \pt_{P,\ul{D_I^\circ}}\to \partial X$ and $c\colon \A_{P,\ul{D_I^\circ}}\times_{\A_{\N}}\pt_{\N}\to \partial X$ be the induced morphisms.
Since the morphism $\pt_{Q,\ul{D_I^\circ}}\to \pt_{P,\ul{D_I^\circ}}$ is a virtual isomorphism in the sense of \cite[Definition 3.2.1]{logsix},
the natural transformation $b_{I*}b_I^*\xrightarrow{\ad} g_{I*}g_I^*$ is an isomorphism.
By Lemma \ref{description.2},
the natural transformation $b_{I*}b_I^*\xrightarrow{\ad} c_{I*}c_I^*$ is an isomorphism.
Hence it suffices to show that the natural transformation $c_{I*}c_I^*\xrightarrow{\ad}h_{I*}h_I^*$ is an isomorphism.
To conclude,
observe that the natural transformation $c_{I*}c_I^*\xrightarrow{\ad} c_{I*}j_*j^*c_I^*$ (resp.\ $h_{I*}h_I^*\xrightarrow{\ad} h_{I*}j_*'j'^*h_I^*$) is an isomorphism by ($ver$-inv) in \cite[Theorem 1.2.1]{logsix} (resp.\ \cite[Proposition 2.4.2]{logsix}).
\end{proof}

\begin{lem}
\label{description.4}
Under the notation of \textup{Construction \ref{description.6}},
there is a natural isomorphism in $\sT(S)$
\[
o^*(M_{\pt_{\N,S}}(W_I^\circ))
\simeq
M_S(F_I^\circ).
\]
\end{lem}
\begin{proof}
It suffices to show $\Psi^{\log}(M_{\G_{m,S}}(\cE_I^\circ-\partial \cE_I^\circ)) \simeq M_{\pt_{\N,S}}(W_I^\circ)$.
For this,
it suffices to show that the natural morphism $M_{\A_{\N,S}}(\cE_I^\circ)\xrightarrow{\ad} j_*j^* M_{\A_{\N,S}}(\cE_I^\circ)$ is an isomorphism,
where $j\colon \G_{m,S}\to \A_S^1$ is the obvious open immersion.
This question is Zariski local on $V$,
so we reduce to the case of Construction \ref{description.7}.

Then the morphism $\cE_I^\circ \to \A_{\N,S}$ is identified with the composite morphism $\A_{\ul{D_I^\circ},\N^r} \to \A_{\N^r,S}
\to
\A_{\N,S}$,
where the first morphism is the obvious one,
and the second morphism is induced by the diagonal map $\N\to \N^r$.
Consider the face $G:=0\oplus \N^{r-1}$ of $\N^r$.
Since the induced map $\N\to (\N^r)_G$ can be identified with the first inclusion $\N\to \N \oplus \Z^{r-1}$,
by \cite[Proposition 2.4.2]{logshriek},
we have an isomorphism $M_{\A_{\N,S}}(\cE_I^\circ)\simeq M_{\A_{\N,S}}(\A_{\N,\ul{D_I^\circ}}\times \G_m^{r-1})$.
Proposition \ref{uni.8} finishes the proof.
\end{proof}

Now, we provide the following description of the functor $o^*$.

\begin{thm}
\label{description.3}
Under the notation of \textup{Construction \ref{description.6}},
there is a natural isomorphism in $\sT(S)$
\begin{equation}
\label{description.3.1}
o^*M_{\pt_{\N,S}}(\partial V)
\simeq
\colim_{I\subset \{1,\ldots,n\},I\neq \emptyset}
M_S(F_I^\circ),
\end{equation}
where the colimit runs over the category of nonempty subsets of $\{1,\ldots,n\}$.
Hence if $V\to X$ is proper,
then there is a natural isomorphism in $\sT(S)$
\begin{equation}
\label{description.3.2}
\Psi(M_{X-\partial X}(V-\partial V))
\simeq
\colim_{I\subset \{1,\ldots,n\},I\neq \emptyset}
M_S(F_I^\circ).
\end{equation}
\end{thm}
\begin{proof}
Let $f_I\colon Z_I\to \pt_{\N,S}$, $g_I\colon Z_I^\circ \to \pt_{\N,S}$,
and $h_I\colon W_I^\circ \to \pt_{\N,S}$
be the induced morphisms.
Using induction on $\lvert I\rvert$ and \cite[Proposition A.0.4]{logSH},
\cite[Corollary 3.6.8]{logshriek} implies that we have a natural isomorphism
\[
f_*f^*
\simeq
\lim_{I\subset \{1,\ldots,n\},I\neq \emptyset}
f_{I*} f_I^*.
\]
Together with Lemmas \ref{description.8} and \ref{description.9},
we have a natural isomorphism
\begin{equation}
\label{description.3.3}
f_*f^*
\simeq
\lim_{I\subset \{1,\ldots,n\},I\neq \emptyset}
g_{I*} g_I^*
\simeq
\lim_{I\subset \{1,\ldots,n\},I\neq \emptyset}
h_{I*} h_I^*.
\end{equation}
By adjunction,
we have a natural isomorphism
\[
f_\sharp f^*
\simeq
\colim_{I\subset \{1,\ldots,n\},I\neq \emptyset}
h_{I\sharp} h_I^*.
\]
Apply $o^*$ to this,
and use Lemma \ref{description.4} to show \eqref{description.3.1}
For \eqref{description.3.2},
use Theorem \ref{lognearby.3}.
\end{proof}

\begin{rmk}
\label{description.5}
Let $X$ be a regular log regular fs log scheme in $\lSch/B$ such that $X$ has a chart $\N$ and $\partial X\simeq \pt_{\N,S}$,
and let $V\to X$ be a vertical log smooth morphism in $\lSch/B$.
As in Proposition \ref{nearby.2},
there exists $n\in \N^+$ such that the projection $V\times_X X'\to X'$ is saturated with $X':=X\times_{\A_\N,u_n}\A_\N$,
where $u_n\colon \A_\N\to \A_\N$ is the morphism induced by the multiplication $n\colon \N\to \N$.
Furthermore,
there exists a dividing cover $V'\to V\times_X X'$ such that $\ul{V'}$ is regular by \cite[10.4]{MR1296725}.
Consider the morphism $w_n\colon \pt_{\N,S}\to \pt_{\N,S}$ induced by $n\colon \N\to \N$.
We have the natural transformations
\begin{align*}
& o^*M_{\partial X}(V\times_X \partial X)
 \simeq
o^*w_n^*M_{\partial X}(V\times_X \partial X)
\\
 \simeq &
o^*M_{\partial X'}(V\times_X \partial X')
\simeq
o^*M_{\partial X'}(W\times_{X'} \partial X'),
\end{align*}
where the first one is due to Theorem \ref{nearby.4}(2),
and the third one is due to ($div$-inv) in \cite[Theorem 1.2.1]{logsix}.
Now,
we can apply Theorem \ref{description.3} to $o^*M_{\partial X'}(W\times_{X'} \partial X')$.
\end{rmk}

\section{Quasi-unipotent motives are motives over the standard log point in characteristic 0}

Throughout this section,
we fix $S\in \Sch/\Q$.

As in \cite[Notation 1.3.24]{MR3381140},
consider the scheme
\[
Q_r(X,f)
:=
X[t,t^{-1},v]/(v^r-ft)
\simeq
X[v,v^{-1}]
\]
over $\G_{m,S}=S[t,t^{-1}]$ for $X\in \Sm/S$, $f\in \Gamma(X,\cO_X^*)$, and $r\in \N^+$.
As in \cite[Definition 1.3.25]{MR3381140},
let $\QUT(S)$ be $\infty$-category of quasi-unipotent motives,
which is the full subcategory of $\sT(\G_{m,S})$ generated under colimits by $Q_r(X,f)$ for $X\in \Sm/S$, $f\in \Gamma(X,\cO_X^*)$, and $r\in \N^+$.

The purpose of this section is to compare $\sT(\pt_{\N,S})$ with $\QUT(S)$,
see Theorem \ref{quasi-uni.5} below.

\begin{df}
\label{quasi-uni.9}
For $X\in \Sm/S$,
$f\in \Gamma(X,\cO_X^*)$,
and a vertical map of sharp fs monoids $\theta\colon \N\to P$,
consider the map $\N\to \Gamma(\A_{P,X},\cM_{\A_{P,X}})$ given by $1\mapsto \theta(1)/f$ using the multiplicative notation for the monoid $\Gamma(\A_{P,X},\cM_{\A_{P,X}})$,
which induces a morphism $\A_{P,X}\to \A_{\N}$ and hence a morphism $\A_{P,X}\to \A_{\N,S}$.
For convenience,
let $\ol{Q}_\theta(X,f)\to \A_{\N,S}$ be this morphism.
Let $Q_\theta(X,f)\to \G_{m,S}$ and $Q_\theta^{\log}(X,f)\to \pt_{\N,S}$ be its pullbacks.

For $r\in \N^+$,
if $\theta$ is the multiplication map $r\colon \N\to \N$,
then $Q_r(X,f)$ coincides with the above one.
\end{df}

\begin{prop}
\label{quasi-uni.10}
For $X\in \Sm/S$, $f\in \Gamma(X,\cO_X^*)$, and a saturated vertical map of sharp fs monoids $\theta\colon \N\to P$,
there is an isomorphism $\ol{Q}_\theta(X,f)\simeq \ol{Q}_\theta(X,1)$ over $\A_{\N,S}$.
\end{prop}
\begin{proof}
Let $F$ be a maximal $\theta$-critical face of $P$.
Since $\theta(\N)\cap F=0$,
the induced map $\theta'\colon \N\to P/F$ is injective.
Hence the induced map $\theta'^\gp \colon \Z\to P^\gp/F^\gp$ is injective.
By \cite[Theorem I.4.8.14(6),(7)]{Ogu},
its cokernel is torsion free,
so there is a retraction $P^\gp/F^\gp\to \Z$ of $\theta'^\gp$.
It follows that there is a retraction $\eta \colon P^\gp\to \Z$ of $\theta^\gp$.

Consider the map $P\to \Gamma(\ol{Q}_\theta(X,f),\cM_{\ol{Q}_\theta(X,f)})$ given by $p\in P\mapsto p/f^{\eta(p)}$ using the multiplicative notation.
This induces an isomorphism $\ol{Q}_\theta(X,f)\xrightarrow{\simeq} \A_{P,X}=\ol{Q}_\theta(X,1)$ over $\A_{\N,S}$.
\end{proof}

\begin{prop}
\label{quasi-uni.1}
For $X\in \Sm/S$, $f\in \Gamma(X,\cO_X^*)$, and a vertical map of sharp fs monoids $\theta\colon \N\to P$,
$\ol{Q}_\theta(X,f)$ is log smooth over $\A_{\N,S}$.
\end{prop}
\begin{proof}
By \cite[Proposition I.4.2.1(4)]{Ogu},
$\theta$ is locally exact since $\N$ is valuative.
Hence by \cite[Theorem I.4.9.1]{Ogu},
there exists $r\in \N^+$ such that the pullback $\eta\colon \N\to Q$ along the multiplication map $r\colon \N\to \N$ is saturated.
Consider the morphism $w_r\colon \A_{\N,S}\to \A_{\N,S}$ induced by $r\colon \N\to \N$,
which is log \'etale by \cite[Corollary IV.3.1.10]{Ogu}.
Using \cite[Theorem 0.2]{zbMATH06164842} and \cite[Proposition D.4]{divspc},
it suffices to show that the projection
\[
\ol{Q}_\eta(X,f)
\simeq
\ol{Q}_\theta(X,f)\times_{\A_{\N,S},w_r}\A_{\N,S}
\to
\A_{\N,S}
\]
is log smooth.
Since $\eta$ is saturated,
Proposition \ref{quasi-uni.10} implies that this can be identified with the morphism $\ol{Q}_{\eta}(X,1)\to \A_{\N,S}$,
which is log smooth by \cite[Theorem IV.3.1.8]{Ogu}.
\end{proof}

\begin{prop}
\label{quasi-uni.11}
For $X\in \Sm/S$, $f\in \Gamma(X,\cO_X^*)$, a vertical map of sharp fs monoids $\theta\colon \N\to P$, and a $\theta$-critical face $G$ of $P$,
the induced morphism
\[
M_{\A_{\N,S}}(\ol{Q}_{\theta_G}(X,f))
\simeq
M_{\A_{\N,S}}(\ol{Q}_\theta(X,f))
\]
in $\sT(\A_{\N,S})$ is an isomorphism,
where $\theta_G\colon \N\to P_G$ is the induced map.
\end{prop}
\begin{proof}
Choose $r\in \N^+$ as in the proof of Proposition \ref{quasi-uni.1}.
Consider the morphism $p_r\colon \A_{\N,S\times \G_m}\simeq \A_{\N\oplus \Z,S}\to \A_{\N,S}$ induced by the map $\N\to \N\oplus \Z$ given by $1\mapsto (r,1)$.
As in the proof of \cite[Proposition 3.3.3]{logGysin},
the functor $p_r^*\colon \sT(\A_{\N,S})\to \sT(\A_{\N,S\times \G_m})$ is conservative.
Hence it suffices to show that the induced morphism in $\sT(\A_{\N,S\times \G_m})$
\[
p_r^*M_{\A_{\N,S}}(\ol{Q}_{\theta_G}(X,f))
\simeq
p_r^*M_{\A_{\N,S}}(\ol{Q}_\theta(X,f))
\]
is an isomorphism.
This can be identified with the induced morphism in $\sT(\A_{\N,S\times \G_m})$
\[
M_{\A_{\N,S\times \G_m}}(\A_{Q_G,X\times \G_m})
\simeq
M_{\A_{\N,S\times \G_m}}(\A_{Q,X\times \G_m})
\]
by Proposition \ref{quasi-uni.10},
where $Q:=P\oplus_{\N,r}\N$,
and $G$ is the face of $Q$ generated by the image of $F$ in $Q$.
We finish the proof by \cite[Proposition 2.4.2]{logshriek}.
\end{proof}

\begin{prop}
\label{quasi-uni.7}
For $X\in \Sm/S$, $f\in \Gamma(X,\cO_X^*)$, and $r\in \N^+$,
the natural morphism in $\sT(\A_{\N,S})$
\[
M_{\A_{\N,S}}(\ol{Q}_r(X,f))
\xrightarrow{ad}
j_* j^* M_{\A_{\N,S}}(\ol{Q}_r(X,f))
\]
is an isomorphism.
\end{prop}
\begin{proof}
Consider the morphism $p_r$ in the proof of Proposition \ref{quasi-uni.11},
whose pullback functor $p_r^*$ is conservative.
It suffices to show that the natural morphism
\[
p_r^*M_{\A_{\N,S}}(\ol{Q}_r(X,f))
\xrightarrow{ad}
p_r^*j_* j^* M_{\A_{\N,S}}(\ol{Q}_r(X,f))
\]
is an isomorphism.
This can be identified with the natural morphism
\[
M_{\A_{\N,S\times \G_m}}(\A_{\N,X\times \G_m})
\xrightarrow{ad}
u_*u^*
M_{\A_{\N,S\times \G_m}}(\A_{\N,X\times \G_m})
\]
using Proposition \ref{quasi-uni.10} and ($\eSm$-BC) in \cite[Theorem 1.2.1]{logsix},
where $u\colon \G_{m,S\times \G_m}\to \A_{\N,S\times \G_m}$ is the obvious open immersion.
Proposition \ref{uni.8} finishes the proof.
\end{proof}

\begin{lem}
\label{quasi-uni.6}
For $X\in \Sm/S$, $f\in \Gamma(X,\cO_X^*)$, and $d,r\in \N^+$,
there exists $Y\in \Sm/S$ and $g\in \Gamma(Y,\cO_Y^*)$ such that
$M_{\A_{\N,S}}(\ol{Q}_r(X,f))$ is a direct summand of $M_{\A_{\N,S}}(\ol{Q}_{rd}(Y,g))$ in $\sT(\A_{\N,S})$.
\end{lem}
\begin{proof}
As in the proof of \cite[Lemma 1.3.33]{MR3381140},
$M_{\G_{m,S}}(Q_r(X,f))$ is a direct summand of $M_{\G_{m,S}}(Q_{rd}(Y,g))$ with $Y:=X[u,u^{-1}]$ and $g:=u^rf$.
Use Proposition \ref{quasi-uni.7} to conclude.
\end{proof}

\begin{prop}
\label{quasi-uni.3}
Let $i\colon \pt_{\N,S}\to \A_{\N,S}$ be the obvious strict closed immersion.
For $X,Y\in \Sm/S$, $f\in \Gamma(X,\cO_X^*)$, $g\in \Gamma(Y,\cO_Y^*)$, and $d,n,r,s\in \N^+$,
the induced morphism
\begin{align*}
&\Hom_{\sT(\A_{\N,S})}(M_{\A_{\N,S}}(\ol{Q}_r(X,f)),M_{\A_{\N,S}}(\ol{Q}_s(Y,g))(d)[n])
\\
\to &
\Hom_{\sT(\pt_{\N,S})}(i^*M_{\A_{\N,S}}(\ol{Q}_r(X,f)),i^*M_{\A_{\N,S}}(\ol{Q}_s(Y,g))(d)[n])
\end{align*}
is an isomorphism.
\end{prop}
\begin{proof}
Compare this with the proof of \cite[Proposition 1.3.34]{MR3381140}.
By Lemma \ref{quasi-uni.6},
we can replace $r$ by $\mathrm{lcm}(r,s)$,
so we reduce to the case when $r=sd$ for some $d\in \N^+$.
The underlying scheme of $\ol{Q}_{sd}(X,f)\times_{\A_{\N,S}}\ol{Q}_s(Y,g)$ is isomorphic to 
\[
(X\times_S Y)[t,v,w,(v^d/w)^{\pm}]/(v^{sd}-ft,w^s-gt),
\]
where the element $(v^d/w)^{\pm}$ comes from the element $(d,-1)\in \N \oplus_{sd,\N,d} \N$.
With $u:=v^d/w$ and $Z:=X\times_S Y[u]/(u^d-f/g)$,
we have an isomorphism
\[
\ol{Q}_{sd}(X,f)\times_{\A_{\N,S}}\ol{Q}_s(Y,g)
\simeq
(Z[v],\N v).
\]
Moreover, its projection to $\ol{Q}_{sd}(X,f)\simeq (X[v],\N v)$ can be identified with
\[
\A_{\N,p}\colon \A_{\N,Z}\to \A_{\N,X},
\]
where $p\colon Z\to X$ is the projection.
Hence by adjunction,
it suffices to show that the induced morphism
\[
\Hom_{\sT(\A_{\N,X})}(\unit,M_{\A_{\N,X}}(\A_{\N,Z})(d)[n])
\to
\Hom_{\sT(\pt_{\N,X})}(\unit,M_{\pt_{\N,X}}(\pt_{\N,Z})(d)[n])
\]
is an isomorphism.
This is a consequence of the proof of Theorem \ref{uni.1}.
\end{proof}

\begin{lem}
\label{quasi-uni.8}
Assume that $S$ is a reduced scheme.
Let $f\colon V\to \pt_{\N,S}$ be a vertical log smooth morphism in $\lSch/B$ such that $V$ has a neat chart $P$.
Consider the closed subset
\[
Z:=\{z\in V:\ol{\cM}_{Z,z}\simeq P\}
\]
of $X$.
We regard $Z$ as a strict closed subscheme of $V$ with the reduced scheme structure.
Then $\ul{Z}$ is smooth over $S$.
\end{lem}
\begin{proof}
By \cite[Proposition IV.17.7.7]{EGA},
the question is strict \'etale local on $X$.
Hence by \cite[Theorem IV.3.3.1(3)]{Ogu} and the assumption that $S$ is a $\Q$-scheme,
we may assume that $f$ admits a neat chart $\N\to P$ such that the induced morphism $X\to \pt_{\N,S}\times_{\A_\N}\A_P$ is strict smooth.
We have the induced cartesian square
\[
\begin{tikzcd}
Z\ar[d]\ar[r]&
\pt_{\N,S}\times_{\A_\N}\pt_P\ar[d]
\\
X\ar[r]&
\pt_{\N,S}\times_{\A_\N}\A_P.
\end{tikzcd}
\]
To conclude,
observe that the underlying scheme of $\pt_{\N,S}\times_{\A_\N}\pt_P$ is isomorphic to $S$.
\end{proof}

\begin{prop}
\label{quasi-uni.12}
Let $i\colon S_\red\to S$ be the obvious closed immersion.
Then the functor
\begin{equation}
\label{quasi-uni.12.1}
i'^*\colon \QUT(S)\to \QUT(S_\red)
\end{equation}
is an equivalence of $\infty$-categories,
where $i'\colon \G_{m,S_\red}\to \G_{m,S}$ is the induced closed immersion.
\end{prop}
\begin{proof}
Since $i'^*\colon \sT(\G_{m,S_\red})\to \sT(\G_{m,S})$ is an equivalence of $\infty$-categories by(Loc) in \cite[Theorem 1.2.1]{logsix},
\eqref{quasi-uni.12.1} is fully faithful.
Hence it suffices to show that \eqref{quasi-uni.12.1} is essentially surjective.

For this,
consider $X\in \Sm/S_\red$, $f\in \Gamma(X,\cO_X^*)$, and $r\in \N^+$.
We need to show that $M_{\G_{m,S_\red}}(Q_r(X,f))$ is in the essential image of \eqref{quasi-uni.12.1}.
This question is Zariski local on $X$.
Hence by \cite[Corollarie IV.17.11.4]{EGA},
we may assume that there exists a factorization $X\xrightarrow{q} S_\red \times \A^n\xrightarrow{p} S_\red$ such that $q$ is \'etale and $p$ is the projection.
Together with \cite[Th\'eor\`eme IV.18.1.2]{EGA},
we see that there exists $Y\in \Sm/S$ such that $Y\times_S S_\red\simeq X$.
We may also assume that $Y$ is affine since we work Zariski locally on $X$.
Now choose any $g\in \Gamma(Y,\cO_Y^*)$ whose image in $\Gamma(X,\cO_X^*)$ is $f$,
which is possible since $Y_\red\simeq X$.
To conclude,
observe that we have an isomorphism $i'^*M_{\G_m,S}(Q_r(Y,g))\simeq M_{\G_{m,S_\red}}(Q_r(X,f))$.
\end{proof}

Ayoub expected that the $\infty$-category of motives over the standard log point is equivalent to the $\infty$-category of rigid analytic motives in characteristic $0$,
which we prove as follows.

\begin{thm}
\label{quasi-uni.5}
Recall that $S$ is in $\Sch/\Q$.
The functor $\Psi^{\log}\colon \sT(\G_{m,S})\to \sT(\pt_{\N,S})$ restricts to an equivalence of $\infty$-categories
\begin{equation}
\label{quasi-uni.5.1}
\Psi^{\log}
\colon
\QUT(S)
\xrightarrow{\simeq}
\sT(\pt_{\N,S}).
\end{equation}
Hence for every field $k$ of characteristic $0$,
there are equivalences of $\infty$-categories
\begin{equation}
\label{quasi-uni.5.2}
\SH(\pt_{\N,k})
\simeq
\RigSH(
k(\! (x) \! )
),
\text{ }
\DA(\pt_{\N,k},\Lambda)
\simeq
\RigDA(
k(\! (x) \! )
,\Lambda),
\end{equation}
where $\Lambda$ is a commutative ring.
\end{thm}
\begin{proof}
We obtain \eqref{quasi-uni.5.2} from Example \ref{lognearby.5}(2), \eqref{quasi-uni.5.1},
and \cite[Scholie 1.3.26]{MR3381140}.
Hence it suffices to show \eqref{quasi-uni.5.1}.

By (Loc) in \cite[Theorem 1.2.1]{logsix},
we have $\sT(\pt_{\N,S})\simeq \sT(\pt_{\N,S_\red})$.
Together with Proposition \ref{quasi-uni.12},
we reduce to the case when $S$ is reduced.

Let $i\colon \pt_{\N,S}\to \A_{\N,S}$ be the obvious strict closed immersion.
By Lemma \ref{uni.2},
the full faithfulness of \eqref{quasi-uni.5.1} is a consequence of Propositions \ref{quasi-uni.7} and \ref{quasi-uni.3}.
For essential surjectivity of \eqref{quasi-uni.5.1},
let $\cC_S$ be the full subcategory of $\sT(\pt_{\N,S})$ generated under colimits by $i^*M_{\A_{\N,S}}(\ol{Q}_r(X,f))(d)[n]$ for $X\in \Sm/S$, $f\in \Gamma(X,\cO_X^*)$, $r\in \N^+$, and $d,n\in \Z$.
We need to show $M_{\pt_{\N,S}}(V)\in \cC_S$ for $V\in \lSm/\pt_{\N,S}$.

The question is Zariski local on $V$,
so we may assume that $\Omega_{V/\pt_{\N,S}}^1$ is a free $\cO_V$-module, $\theta\colon \N\to P$ is a chart of $V\to \pt_{\N,S}$,
and $\ol{P}$ is a neat chart of $V$ at a point $v\in V$.
We proceed by induction on $r:=\mathrm{rank}(\ol{P}^\gp)$.
The claim is clear if $r=0$ since $V=\emptyset$ in this case.
Assume $r\geq 1$.

Consider the strict closed subscheme $Z$ of $V$ in Lemma \ref{quasi-uni.8}.
Note that $\ul{Z}$ is smooth over $S$.
Let $h\colon Z\to \pt_{\N,S}$ be the induced morphism.
By \cite[Theorem 1.3.1(3)]{logsix},
there is a natural isomorphism $p_\sharp \simeq p_!(r-1)[2(r-1)]$ since $\Omega_{V^1}/\pt_{\N,S}$ is free,
where $p\colon V\to \pt_{\N,S}$ is the structure morphism.
Together with the (Loc) in \cite[Theorem 1.2.1]{logsix},
we have a fiber sequence
\[
M_{\pt_{\N,S}}(V-Z)
\to
M_{\pt_{\N,S}}(V)
\to
h_!h^*\unit(r-1)[2(r-1)].
\]
By induction,
we have $M_{\pt_{\N,S}}(V-Z)\in \cC_S$.
Hence it suffices to show $h_!h^*\unit\in \cC_S$.

Consider the induced map $\ol{\theta}\colon \N\to \ol{P}$.
We set $W:=Q_\theta^{\log}(\ul{Z},g)$,
where $g$ is the restriction of $f$ to $\ul{Z}$.
Observe that $W$ contains $Z$ as a strict closed subscheme.
As above,
we have a fiber sequence
\[
M_{\pt_{\N,S}}(W-Z)
\to
M_{\pt_{\N,S}}(W)
\to
h_!h^*\unit(r-1)[2(r-1)],
\]
where $(r-1)[2(r-1)]$ comes from the fact that $\Omega_{W/\pt_{\N,S}}^1$ is free, see \cite[Proposition IV.1.1.4(2)]{Ogu}.
By induction,
we have $M_{\pt_{\N,S}}(W-Z)\in \cC_S$.
Hence it suffices to show $M_{\pt_{\N,S}}(W)\in \cC_S$.

This holds if $r=1$ by the definition of $\cC_S$,
so assume $r>1$.
Let $F$ be a maximal $\theta$-critical face of $\ol{P}$.
By Proposition \ref{quasi-uni.11},
there is an isomorphism
$M_{\pt_{\N,S}}(W)\simeq M_{\pt_{\N,S}}(Q_{\theta_F}^{\log}(\ol{Z},g))$,
where $\theta_F\colon \N\to P_F$ is the induced map.
To conclude,
observe that we have $M_{\pt_{\N,S}}(Q_{\theta_F}^{\log}(\ol{Z},g))\in \cC_S$ by induction.
\end{proof}

\begin{rmk}
We do not know whether \eqref{quasi-uni.5.2} holds or not if $S$ is not a scheme over $\Q$.
\end{rmk}

\begin{prop}
\label{quasi-uni.13}
The triangle
\[
\begin{tikzcd}
\QUT(S)\ar[r,"\Psi^{\log}"]\ar[rd,"i_1^*"']&
\sT(\pt_{\N,S})\ar[d,"o^*"]
\\
&
\sT(S)
\end{tikzcd}
\]
commutes,
where $i_1\colon S\to \G_{m,S}$ is the $1$-section.
\end{prop}
\begin{proof}
For $r\in \N^+$,
let $\rQ_r\UT(S)$ denote the full subcategory of $\sT(\G_{m,S})$ generated under colimits by $Q_s(X,f)$ for $X\in \Sm/S$, $f\in \Gamma(X,\cO_X^*)$, and $r\in \N^+$ with $s\vert r$.
As in Proposition \ref{nearby.15},
we have an equivalence of $\infty$-categories
\begin{equation}
\label{quasi-uni.13.1}
\colimPrL_{r\in (\N^+,\times)}
\rQ_r\UT(S)
\simeq
\QUT(S).
\end{equation}
Let $w_r\colon \G_{m,S}\to \G_{m,S}$ and $w_r\colon \pt_{\N,S}\to \pt_{\N,S}$ be the restrictions of $w_r\colon \A_{\N,S}\to \A_{\N,S}$ induced by the multiplication $r\colon \N\to \N$.
Observe that $w_r\colon \A_{\N,S}\to \A_{\N,S}$ is log \'etale by \cite[Corollary IV.3.1.10]{Ogu} since $S$ is a scheme over $\Q$.
It is also vertical.
The functor $w_r^*\colon \QUT(S)\to \QUT(S)$ restricts to $\rQ_r\UT(S)\to \UT(S)$.
Hence by Proposition \ref{lognearby.7},
we have a commutative square
\[
\begin{tikzcd}
\rQ_r\UT(S)\ar[r,"\Psi^{\log}"]\ar[d,"w_r^*"']&
\sT^{\sat_r}(\pt_{\N,S})\ar[d,"w_r^*"]
\\
\UT(S)\ar[r,"\Psi^{\log}"]&
\sT^{\st}(\pt_{\N,S}).
\end{tikzcd}
\]
Compose this square with $o^*\colon \sT^\st(\pt_{\N,S})\to \sT(S)$ and recall Definition \ref{uni.4} to obtain a commutative triangle
\[
\begin{tikzcd}
\rQ_r\UT(S)\ar[r,"\Psi^{\log}"]\ar[rd,"i_1^*w_r^*"']&
\sT(\pt_{\N,S})\ar[d,"o^*"]
\\
&
\sT(S).
\end{tikzcd}
\]
Observe that we have $i_1^*w_r^*\simeq i_1^*$.
Take colimits along $r\in (\N^+,\times)$ and use \eqref{quasi-uni.13.1} to obtain the desired commutative square.
\end{proof}

\section{Comparison with Ayoub's motivic nearby cycles functors in characteristic 0}

In this section,
we show that $\Psi$ in Definition \ref{nearby.14} agrees with Ayoub's motivic nearby cycles functor $\Psi^{\Ayo}$ in the equal characteristic $0$ case below.

\begin{thm}
\label{comp.1}
Let $k$ be a field of characteristic $0$, and let $X$ be the spectrum of a DVR with residue field $k$ and the standard log structure such that there is a strict morphism $f\colon X\to \A_{\N,k}$.
Then there is a natural isomorphism
\[
\Psi^{\Ayo}\simeq \Psi
\colon
\sT(X-\partial X)
\to
\sT(\ul{\partial X}).
\]
\end{thm}
\begin{proof}
Let $\bDelta$ be the simplex category.
Consider the diagram
\begin{equation}
\label{comp.1.1}
\bDelta \times (\N^+,\times) \to \Sch/\G_{m,k}
\end{equation}
in \cite[D\'efinition 3.5.3]{Ayo07},
which sends $(n,r)\in \bDelta \times (\N^+,\times)$ to $Q_r(\G_{m,k}^n,1)$.
For $\alpha:=(n,r)$,
form the cartesian square
\[
\begin{tikzcd}
V_\alpha\ar[d,"\eta_\alpha"']\ar[r]&
Q_r(\G_{m,k}^n,1)\ar[d,"\theta_\alpha"]
\\
X-\partial X\ar[r,"g"]&
\G_{m,k},
\end{tikzcd}
\]
where $\theta_\alpha$ is induced by \eqref{comp.1.1},
and $g$ is induced by $f\colon X\to \A_{\N,k}$.
Let $i\colon \partial X\to X$ and $j\colon X-\partial X\to X$ be the obvious strict immersions.
We have a natural isomorphism
\[
\Psi^\Ayo(\cF)
\simeq
\colim_\alpha \ul{i}^*\ul{j}_*(\eta_{\alpha*}\unit \otimes \cF)
\]
as in \cite[(1.103)]{MR3381140}.
Consider the composite functor
\[
\gamma
\colon
\sT(X-\partial X)
\xrightarrow{\Psi^{\log}}
\sT(\pt_{\N,k})
\xrightarrow{\simeq}
\QUT(k),
\]
where the second functor is an equivalence due to Theorem \ref{quasi-uni.5}.
Let $q\colon \G_{m,k}\to \Spec(k)$ be the projection,
which induces $q_*\colon \QUT(k)\to \sT(k)$.
We have $q_*\gamma\simeq \ul{i}^*\ul{j}_*$ using Propositions \ref{lognearby.6} twice,
so we have a natural isomorphism
\[
\Psi^\Ayo(\cF)
\simeq
\colim_\alpha q_*\gamma(\eta_{\alpha*}\unit \otimes \cF).
\]
Since $\gamma$ is symmetric monoidal by Theorem \ref{lognearby.4},
we have a natural isomorphism
\[
\Psi^\Ayo(\cF)
\simeq
\colim_\alpha q_*(\gamma\eta_{\alpha*}\unit \otimes \gamma \cF).
\]
Theorem \ref{lognearby.3} yields a natural isomorphism $\gamma\eta_{\alpha*}\unit \simeq \theta_{\alpha*}\unit$.
Hence we have a natural isomorphism
\[
\Psi^\Ayo(\cF)
\simeq
\colim_\alpha q_*(\theta_{\alpha*}\unit \otimes \gamma \cF).
\]
Let $i_1\colon \Spec(k)\to \G_{m,k}$ be the $1$-section,
which induces $i_1^*\colon \QUT(\G_{m,k})\to \sT(k)$.
Together with the natural isomorphism
\[
i_1^*
\simeq
\colim_\alpha q_*(\theta_{\alpha*}\unit\otimes (-))
\]
in \cite[(1.107)]{MR3381140},
we have a natural isomorphism $\Psi^{\Ayo}\simeq i_1^* \gamma$.
Proposition \ref{quasi-uni.13} finishes the proof.
\end{proof}

\begin{rmk}
\label{comp.2}
We do not know whether Theorem \ref{comp.1} holds or not if the characteristic of $k$ is positive.
The proof uses the assumption of characteristic $0$ several times.

Let us review what is known in the direction of the comparison between motives over the standard log point and rigid analytic motives with the \'etale topology and $\Q$-coefficients.
Let $\cO_C$ be a valuation ring with residue field $k$ of positive characteristic and fraction field $C$ such that $C$ is the completion of an algebraic closure of a discrete valuation field,
let $\cO_C^\dagger$ be the log ring $(\cO_C,\cO_C-\{0\})$,
and let $k^\dagger$ be the log ring $(k,(\cO_C-\{0\})\oplus_{\cO_C^*}k^*)$.

(1) 
Consider the category $\lSm^\mathrm{ss}/k^\dagger$ of semistable log smooth schemes over $k^\dagger$ in the sense of \cite[Definition 2.11(1)]{2207.00369}.
Binda, Kato, and Vezzani \cite[Theorem 2.22, Proposition 2.34]{2207.00369} showed that $\RigDA_{\et}(C,\Q)$ is a localization of the $\infty$-category
\[
\Sp_{\P^1}((\A^1)^{-1}\Sh_{\setale}^\wedge (\lSm^\mathrm{ss}/k^\dagger,\rD(\Q)))
\]
of $\P^1$-spectra on the $\infty$-category $\A^1$-local strict \'etale hypersheaves of $\Q$-complexes on $\lSm^\mathrm{ss}/k^\dagger$.

(2) 
By \cite[Theorem 3.7.21]{AGV} (see also \cite[Remark 4.6]{2306.05099}), \cite[Corollary 4.12]{2306.05099}, and Theorem \ref{uni.1},
there is an equivalence of $\infty$-categories
\[
\Mod_{\rM \Q}^\st(k^\dagger)
\simeq
\RigDA_{\et}(C,\Q).
\]
\end{rmk}

\bibliography{bib}
\bibliographystyle{siam}

\end{document}